\documentclass{amsart}
\usepackage[T1]{fontenc}
\usepackage[utf8]{inputenc}
\usepackage[english]{babel}

\usepackage[margin=1in]{geometry}
\usepackage{amsmath}
\usepackage{amsthm}
\usepackage{amssymb}
\usepackage{paralist}
\usepackage{graphicx}
\usepackage{verbatim}
\usepackage{colonequals} 
\usepackage{hyperref}

\theoremstyle{definition}
\newtheorem{defi}{Definition}[section]
\newtheorem{es}[defi]{Example}

\newtheorem{oss}[defi]{Remark}
\theoremstyle{plain}
\newtheorem{thm}[defi]{Theorem}
\newtheorem*{nthm}{Theorem}
\newtheorem*{con}{Conjecture}
\newtheorem{que}{Question}

\newtheorem{pro}[defi]{Proposition}
\newtheorem{lemma}[defi]{Lemma}
\newtheorem{cor}[defi]{Corollary}

\newcommand{\N}{\mathbb{N}}

\newcommand{\Z}{\mathbb{Z}}
\newcommand{\Q}{\mathbb{Q}}
\newcommand{\R}{\mathbb{R}}

\newcommand{\C}{\mathbb{C}}
\newcommand{\cC}{\mathcal{C}}
\newcommand{\cD}{\mathcal{D}}

\newcommand{\A}{\mathbb{A}}

\newcommand{\pr}{\mathbb{P}}
\newcommand{\cO}{\mathcal{O}}

\newcommand{\virg}[1]{\textquotedblleft #1\textquotedblright}

\title{Triple intersections for algebraic curves on the real torus}
\author{
Danilo Avaro
}
\address{Dipartimento di Matematica Guido Castelnuovo, Sapienza Università di Roma, Piazzale Aldo Moro 5, I-00185 Roma}
\email{danilo.avaro@uniroma1.it}
\subjclass[2020]{14H50, 14G05, 11G30, 37P05}
\keywords{Integral points, algebraic curves, arithmetic dynamics}
\usepackage{microtype}
\begin{document}

\begin{abstract}
    Motivated by a question of Corvaja and Zannier, we investigate whether the projections on the real torus $\R^2/\Z^2$ of three real algebraic plane curves admit infinitely many triple intersections. We analyze the case of three lines in detail, obtaining a complete classification. For higher-degree curves, we derive several finiteness results using Siegel’s and Levin’s theorems on integral points on curves. On the other hand, we construct explicit examples of higher-degree curves with infinitely many triple intersections.
\end{abstract}

\maketitle

\section{Introduction}

Given three algebraic plane curves defined over the real numbers, one expects that their intersection is empty. However, their projections onto the real torus $\R^2/\Z^2$ may intersect. In this paper, we answer several cases of the following question concerning the finiteness of triple intersections on the real torus, posed by Corvaja and Zannier in~\cite{CZ2023}.

\begin{que}\label{QuestionTorus}
    Given three algebraic plane curves in $\A^2(\R)$, determine if there exist infinitely many $\Z^2$-orbits of points of $\A^2(\R)$ intersecting each of the curves. 
\end{que}

Question~\ref{QuestionTorus} can also be formulated in terms of points of $\A^2(\R)$ instead of orbits. Indeed, given three algebraic plane curves $\cC_1,\cC_2,\cC_3$, we can ask if there exist infinitely many points $P\in \cC_3(\R)$ together with vectors $v_1,v_2\in\Z^2$ depending on $P$ satisfying $P+v_1\in \cC_1(\R)$ and $P+v_2\in\cC_2(\R).$ 
The two formulations are not generally equivalent. Indeed, infinitely many points on $\cC_3$ may belong to the same orbit. 

The problem also admits an interpretation in the framework of unlikely intersections. Let us denote by $\tau_v(\cC)$ the translation of the curve $\cC$ by the integral vector $v\in \Z^2$. The set 
$$\mathcal{A}=\bigcup_{v_1,v_2\in\Z^2}\tau_{v_1}(\cC_1)\cap \tau_{v_2}(\cC_2)$$
is a countable union of finite sets. Then, for a general curve $\cC_3$, we expect the set $\cC_3\cap \mathcal{A}$ to be finite. 

We remark that, over the real numbers, the problem is already non-trivial for pairs of curves. For instance, suppose that two irreducible algebraic curves have compact images on the real torus. Then, their projections intersect only finitely many times (Theorem~\ref{compactintersections}). This motivates the study of topological properties of projections of real algebraic curves on the real torus, exposed in Section~\ref{secproj}. Moreover, in Section~\ref{secpairs} we address the case of pairs of curves.

The following theorem summarizes the main results of this paper. We actually prove slightly more general results, but we state them here in a simplified form for clarity of exposition.    

\begin{nthm}
    Let $\cC_1,\cC_2,\cC_3$ be three irreducible algebraic plane curves with degree $d_1\leq d_2\leq d_3$ respectively. In the following cases, there are only finitely many triple intersections on the real torus.    
    \begin{compactenum}[(a)]
        \item Two of the three curves have compact images on the real torus (see Theorem~{\normalfont\ref{compactintersections}}).
        \item The curves are lines, i.e., $d_1=d_2=d_3=1$, and their slopes do not satisfy a specific algebraic relation over $\Q$ (see Theorem~{\normalfont\ref{infinitelines}}).
        \item $d_1=d_2=1$, and $\cC_3$ is hyperbolic (see Definition~{\normalfont\ref{defhyp}} and Theorem~{\normalfont\ref{Twolinesandhyperbolic}}).
        \item The gonality of $\cC_3$ is greater than $d_1d_2$ (see Proposition~{\normalfont\ref{conLevin}}).
        \item $d_1=1,$  $d_2=d_3>2$, at least one between $\cC_2$ and $\cC_3$ has gonality greater than $\lfloor\frac{d}{2}\rfloor$, and there exists no curve $\cC$ such that $\cC_1$ and $\cC_2$ are mapped to $\cC$ via an endomorphism of $\A^2$ (see Theorem~{\normalfont\ref{lineandsamedegree}}).
    \end{compactenum}
\end{nthm}

Our methods will primarily involve studying integral points on algebraic curves. More precisely, for proving finiteness results, we will apply Siegel's Theorem~\ref{SiegelThm} on integral points on curves, a generalization due to Levin (Theorem~\ref{Levin}) for integral points of higher degree, and a theorem of Bilu and Tichy on Diophantine equations of the form $f(x)=g(y)$ where $f,g$ are polynomials with algebraic coefficients (Theorem~\ref{BiluTichy}). Since these results require the curves to be defined over number fields, in Section~\ref{Secspecialization} we provide a specialization argument that reduces the general case to the case of curves defined over number fields.

We construct examples of triples of curves admitting infinitely many orbits intersecting all of them in the following cases, which also show that some hypotheses of our theorems cannot be removed:
\begin{compactenum}[(i)]
    \item Three lines whose slopes satisfy some prescribed dependence relation over $\Q$ (see Theorem~\ref{infinitelines}).
    \item Two lines and a curve of genus zero of arbitrary degree and either one or two points at infinity (See Examples~\ref{onepoint},~\ref{twopoints} and Theorems~\ref{mainparabolas},~\ref{mainhyperbolas}).
    \item Three conics (See Examples~\ref{threehyperbolas},~\ref{threeparabolas}).
    \item One line, one rational curve of degree $d$ and a curve of positive genus with gonality $d$ (See Example~\ref{smallgon}).
    \item One line and two cubic curves which are sent to the same curve by two endomorphisms of $\A^2$ (See Example~\ref{linetwocub}).
\end{compactenum}

Our techniques do not appear to extend to the remaining cases of Question~\ref{QuestionTorus}. For instance, when the three curves have the same degree $d>2$, and at least two of them have non-compact images on the real torus, we do not have a conjectural answer.

\subsection{Dynamical aspects of Question~\ref{QuestionTorus}}

We can rewrite Question~\ref{QuestionTorus} in the setting of the dynamical Mordell-Lang problem.

\begin{que}[The dynamical Mordell-Lang problem {\cite[Question 3.6.0.1]{BGT2016}}]\label{MLproblem}
    Let $X$ be a quasi-projective variety defined over $\mathbb{C}$, let
    $\Phi_1, \ldots, \Phi_r$ be commuting endomorphisms of $X$, let
    $\alpha \in X(\mathbb{C})$, and let $V \subseteq X$ be a subvariety.
    Is it true that the set of tuples
    $$ (n_1, \ldots, n_r) \in \mathbb{N}_0^r\quad \text{satisfying}\quad \Phi_1^{n_1} \cdots \Phi_r^{n_r}(\alpha) \in V $$
    is a union of at most finitely many sets of the form
    $$ \gamma + (H \cap \mathbb{N}_0^r), $$
    where $\gamma \in \mathbb{N}_0^r$ and $H \subseteq \mathbb{Z}^r$ is a subgroup?
\end{que}

Every algebraic plane projective curve of degree $d>0$ can be identified (as a set) as the zeroes of a homogeneous polynomial in three variables of degree $d$. Fixing an affine plane $\A^2\subset \pr^2$, the set of affine curves of degree $d$ can be identified as an open subset $\mathcal{U}_d\subset \pr^{\frac{d(d+3)}{2}}$, since they are the curves not containing the line at infinity. Given two integers $d, e$, we can consider the following action of $\Z^4$ on $X:=\mathcal{U}_d\times \mathcal{U}_e$
$$\Z^4\times (\mathcal{U}_d\times \mathcal{U}_e)\rightarrow \mathcal{U}_d\times \mathcal{U}_e\colon ((h,k,m,n),(\cC_1,\cC_2))\mapsto (\tau_{(h,k)}(\cC_1),\tau_{(m,n)}(\cC_2)).$$
Moreover, given an affine plane curve $\cC$, we associate to it the subvariety $Y\subset X$ defined as 
$$Y_\cC:=\left\{(\cC_1,\cC_2)\in X\mid \cC_1\cap\cC_2\cap\cC\neq \emptyset\right\}.$$ 
If we fix three affine plane curves $\cC_1,\cC_2,\cC_3$ with $\deg \cC_1=d>0$, $\deg \cC_2= e>0$, we fix a point $(\cC_1,\cC_2)\in X$ and a subvariety $Y_{\cC_3}\subset X$. With this correspondence, proving the finiteness of the $\Z^2$-orbits intersecting the curves $\cC_1,\cC_2,\cC_3$, we trivially obtain that the answer to the dynamical Mordell-Lang problem is affirmative provided that every orbit intersects the curves only finitely many times. For instance, the two formulations are equivalent if the three curves are hyperbolic.  

\subsection{Related questions}

Question~\ref{QuestionTorus} is a particular case of the following question, raised by Corvaja and Zannier in~\cite{CZ2023}.

\begin{que}\label{QuestionCZ}
    Given an algebraic surface with a finitely generated commutative semigroup $\Gamma$ of
    (rational) endomorphisms, and given three curves on the surface, what can be said assuming the existence of infinitely many $\Gamma$-orbits intersecting each of the curves?
\end{que}

In the same paper, they address Question~\ref{QuestionCZ} when the surface is an elliptic surface, that is, a surface fibred in elliptic curves, and $\Gamma$ is generated by the translation by a section of the elliptic fibration, relating this issue to finiteness results for elliptical billiards. Moreover, given a linear automorphism $\beta\in\mathrm{Aut}(\pr^2)$, they classify triples of lines in $\pr^2(\C)$ for which there exist infinitely many $\beta$-orbits that intersect all of them (\cite[Proposition 4.1]{CZ2023}).

They also wonder about a dynamical explanation for the existence of infinitely many orbits. Indeed, they proved that whenever there are infinitely many orbits intersecting the three lines, they are contained in a finite number of orbits for an endomorphism of $\pr^2$ that commutes with $\beta$.

Following this dynamical perspective, we also seek dynamical reasons for the existence of infinitely many orbits that intersect all three curves. However, we cannot provide a uniform explanation for all our examples. We will prove in Theorem~\ref{infSaction} that under suitable hypotheses, there are infinitely many orbits for the action of a non-finitely generated semigroup of endomorphisms of $\A^2$ that intersect three affine lines. Conversely, Theorem~\ref{lineandsamedegree} implies that, if infinitely many orbits intersect one line and two curves of the same degree $d>2$, then the problem trivializes after enlarging the acting semigroup by two additional endomorphisms. 

\section*{Acknowledgments}

This work originated from the author's Master's thesis, completed under the supervision of Pietro Corvaja. The author is deeply grateful to him for his guidance and for his many valuable suggestions and comments on earlier drafts of this paper. The author would also like to thank Amos Turchet for many stimulating discussions on this topic, Alessio Caminata, Giacomo Bortolussi, and Ilaria Cruciani. This work was partially supported by the \virg{National Group for Algebraic and Geometric Structures, and their Applications} (GNSAGA - INdAM) and by Sapienza University of Rome, project number B83C25004300005.

\section{Projection of algebraic curves onto the real torus}\label{secproj}

In this section, we classify the cases in which the projection of an algebraic plane curve on the real torus is compact. In particular, we prove the following theorem as a consequence of Propositions~\ref{irrasym},~\ref{asymnotcontained},~\ref{noasym}.

\begin{thm}\label{compactproj}
    Let $\cC\subset \R^2$ be a real irreducible algebraic curve. The projection of $\cC$ on the real torus is compact if and only if $\cC$ is compact or $\cC$ is a line with a rational slope. 
\end{thm}

Theorem~\ref{compactproj} is applied to Question~\ref{QuestionTorus} in Theorem~\ref{compactintersections}. Indeed, we prove that if two irreducible curves not in the same orbit for the action of $\Z^2$ have both compact images on the real torus, then they intersect only finitely many times on the real torus.

Our arguments rely on the study of the asymptotic behavior of algebraic curves via Puiseux series as explained in~\cite{BP2014}, and on Weyl's Theorem on uniformly distributed polynomial sequences. Throughout this section, we will denote by $\pi \colon \mathbb{R}^2 \rightarrow \mathbb{R}^2 / \mathbb{Z}^2$ the canonical projection onto the real torus.

\subsection{Asymptotic behavior of algebraic curves}

In this section, following~\cite{BP2014}, we recall the notion of \emph{infinity branch} and describe the behavior of algebraic plane curves at infinity. 

\begin{defi}
    A \emph{Laurent series} over $\C$ is a formal sum $\sum_{i\geq r}a_it^i$ where $r\in\Z$ and $a_i\in\C$ for all $i\geq r$. The field of the Laurent series is denoted by $\C((t))$. The field of the \emph{Puiseux series} over $\C$ is the set 
    $$\mathrm{Puis}_\C(t)=\bigcup_{n=1}^\infty \C((t^{1/n}))$$
    endowed with the natural operation of sum and product of series. Given a Puiseux series $\phi(t)$, the minimum integer $n$ such that $\phi(t^n)\in\C((t))$ is called the ramification index of $\phi$ and is denoted by $e(\phi)$.
\end{defi}

\begin{thm}[Newton's Theorem]
    The field $\mathrm{Puis}_\C(t)$ is algebraically closed.
\end{thm}

\begin{defi}
Let $\cC$ be a real plane algebraic curve over $\C$ defined implicitly by the irreducible polynomial $f(x, y) \in \mathbb{R}[x, y]$. Let $A$ and $B$ be series in $\mathbb{C}((t))$ such that: 
\begin{itemize}
    \item[(i)] $f(A(t), B(t)) = 0$ where both series converge, 
    \item[(ii)] not both $A$ and $B$ are constants.
\end{itemize}
Then $P = (A, B)$ is called an \emph{(affine) local parametrization} of $\cC$. Moreover, if $A, B$ are both regular at 0, the point $P(0) = (a, b) \in \cC$ is called the center of $P$.
\end{defi}

\begin{defi}
     Given a local parametrization $(X, Y)$ of a plane curve $\cC$, the set of all points $(X(t), Y (t))$ obtained by allowing $t$ to vary within some neighborhood of $0$ where $X(t)$ and $Y(t)$ converge is called a \emph{branch} of $\cC$.
\end{defi}

Let $\cC$ be the irreducible plane curve defined by the polynomial $f(x,y)\in\R[x,y]$. Let $\widetilde{\cC}$ be the projective curve defined by the homogeneous polynomial $F(X, Y, Z)=0$ which satisfies $F(x,y,1)=f(x,y)$ and $\deg F=\deg f$. Let $P=(a:b:0)$ be a point at infinity of $\widetilde{\cC}$, we suppose $a\neq 0$ so we can set $P=(1:m:0)$.

We parameterize the curve in an affine chart that contains $P$. We introduce two new variables
$$u:=\frac{Y}{X}\qquad v:=\frac{Z}{X},$$
 define $g(u,v):=F(1:u:v)$, and compute a solution $(\phi(t),t)$ of $g(u,v)=0$ in $\mathrm{Puis}_\C(t)$. Then we have $g(\phi(t),t)=0$ in a neighborhood of $t=0$ where $\phi(t)$ converges, that is, there exist $M\in\R^+$ such that
$$F(1:\phi(t):t)=0 \quad \text{for }t\in\C,|t|<M.$$
We get 
$$f(t^{-1},t^{-1}\phi(t))=F(t^{-1}:t^{-1}\phi(t):1)=0\quad\text{for }t\in\C,|t|<M,$$
then defining $z:=t^{-1}$ and $r(z)=z\phi(z^{-1})$ we obtain
$$f(z,r(z))=0\quad\text{for }z\in\C,|z|>M^{-1}.$$

We can write $\phi(t)=m+\sum_{i=1}^{\infty}a_it^{n_i/n}$ where $a_i\in\C^*$, $n,n_i\in\N$ and $n_i<n_{i+1}$ for all $i\in\N$. Then we can write 
$$r(z)=mz+\sum_{i=0}^\infty a_iz^{1-n_i/n}.$$ 
In particular, there are only finitely many terms with non-negative exponent. We remark that, if $m\in\R$, that is, the point at infinity is real, then $r(z)$ has real coefficients. 

\begin{defi}
    Let $f(t)=\sum_{i\geq r}a_ix^{i/n}\in\mathrm{Puis}_\C(t)$ be a Puiseux series with ramification index $n$. The \emph{conjugates} of $f$ are the Puiseux series $\sigma_\varepsilon(f)$ defined by 
    $$\sigma_\varepsilon(f)=\sum_{i\geq r}\varepsilon^ia_ix^{i/n}$$
    where $\varepsilon$ is a complex number satisfying $\varepsilon^n=1$. The set of all conjugates of $f$ is called \emph{conjugacy class} of $f$.
\end{defi}

Since $e(\phi)=n$, there are $n$ different series in its conjugacy class. Let $\phi_1,\dots,\phi_n$ be these series and define $r_i(z)=z\phi_i(z^{-1})$, which converges for $|z|>M_i$. 

\begin{defi}
    The set $B = \bigcup_{i=1}^N L_i$ where
$$ L_i = \left\{(z, r_i(z)) \in \mathbb{C}^2 : z \in \mathbb{C}, |z| > \max\{M_1,\ldots,M_n\}\right\}$$
is called an \emph{infinity branch} of the affine plane curve $\cC$. The subsets $L_1, \ldots, L_N$ are called the \emph{leaves} of the infinity branch $B$.

\end{defi}

\begin{defi}
    Given two leaves, $L = \{(z, r(z)) \in \mathbb{C}^2 : z \in \mathbb{C}, |z| > M\}$ and $L' = \{(z, r'(z)) \in \mathbb{C}^2 : z \in \mathbb{C}, |z| > M\}$, we say that they are convergent if $\lim_{z \to \infty} (r(z) - r'(z)) = 0$.  Two infinity branches, $B$ and $B'$, are convergent if there
exist two convergent leaves $L \subseteq B$ and $L' \subseteq B'$.
\end{defi}

\begin{pro}\label{convergent}
    Two leaves 
    $$L = \{(z, r(z)) \in \mathbb{C}^2: z \in \mathbb{C}, |z| > M\}\quad\text{and}\quad L' = \{(z, r'(z)) \in \mathbb{C}^2: z \in \mathbb{C}, |z| > M\}$$ are convergent if and only if the terms with non-negative exponent in the series $r(z)$ and $r'(z)$ are the same.
\end{pro}

\subsection{Dense curves in the real torus}

In this section, we prove several density results concerning the projections of algebraic curves onto the real torus. We begin by studying curves with asymptotes, relating this case to the following well-known result about affine lines.

\begin{pro}\label{denselines}
    Let $\ell\subset \R^2$ be an affine line. Then, the projection $\pi$ restricted to $\ell$ is injective if and only if $\ell$ has irrational slope. Moreover, if $\ell$ has an irrational slope, then for each half-line $\ell'\subset \ell$, the projection $\pi(\ell')$ is dense in the torus; otherwise, $\pi(\ell)$ is compact and closed.  
\end{pro}

\begin{pro}\label{irrasym}
    Let $\cC\subset\R^2$ be a real algebraic curve with an asymptote $\ell$ with irrational slope. Then $\pi(\cC)$ is dense in the real torus. 
\end{pro}
\begin{proof}
    We fix a point $(x,y)\in\R^2/\Z^2$ and a real number $\varepsilon>0$. Then there is an half-line $\ell'\subset \ell$ and an infinity branch $B$ of $\cC$ such that for each point $P\in B$ the distance of $P$ from $\ell'$ is less than $\frac{\varepsilon}{2}$. From Proposition~\ref{denselines}, there exists $Q\in \ell'$ such that the distance from $Q$ to $(x,y)$ on the torus is less than $\frac{\varepsilon}{2}$. Therefore, there exists a point of $B$ which has distance at most $\varepsilon$ from $(x,y)$, that is, $\cC$ is dense on the torus. 
\end{proof}

\begin{pro}\label{ratasym}
    Let $\cC\subset\R^2$ be an irreducible real algebraic curve such that all infinity branches have an asymptote $\ell$ with rational slope. Then $\pi(\cC)$ is not dense in the real torus. 
\end{pro}
\begin{proof}
    Let $\ell_1,\ldots, \ell_k$ be the asymptotes. For each $\varepsilon>0$ we define the sets 
    $$B_{\varepsilon,k}=\left\{(x,y)\in\R^2\mid d((x,y),\ell_k)\leq \varepsilon\right\}.$$
    From Proposition~\ref{convergent}, we can choose $R>0$ such that 
    $$C\subset [-R,R]^2\cup B_{\varepsilon,1}\cup\cdots\cup B_{\varepsilon,k}.$$
    The set $C\cap [-R, R]^2$ is compact and then has a compact image on the real torus, and the sets $\pi(B_{\varepsilon,k})$ are compact subsets of the torus. Therefore, the set
    $$D_\varepsilon:=\bigcup_{i=1}^k \pi(B_{\varepsilon,i})\cup \pi(\cC\cap [-R,R]^2)\subseteq \R^2/\Z^2.$$
    is compact and hence closed. If the set $D$ is strictly contained in the real torus, then $\pi(\cC)$ cannot be dense. 
      
    Now we prove that there exists $\varepsilon>0$ such that $D_\varepsilon\neq \R^2/\Z^2$. Let $\ell$ be an affine line such that all the lines in the orbit of $\ell$ under the action of $\Z^2$ are distinct from $\ell_1,\ldots,\ell_k$. Then, each line in the orbit of $\ell$ under the action of $\Z^2$ intersects $\ell_1,\ldots,\ell_k,$ and $\cC$ in finitely many points. Since there are countably many lines in the orbit of $\ell$, there are at most countably many intersections between $\pi(\ell)$ and 
    $$\pi(\ell_1)\cup\ldots \pi(\ell_k)\cup \pi(\cC).$$
    Then, there exists 
    $$P\in \R^2/\Z^2\setminus [\pi(\ell_1)\cup\ldots \cup\pi(\ell_k)\cup \pi(\cC)].$$
    For each $i=1,\ldots,k$, $P$ has positive distance $d_k$ from $\ell_k$, then it is enough to choose 
    $$\varepsilon<\frac{1}{2}\min_{k}\{d_k\}.$$\qedhere
\end{proof}

Under the hypothesis of Proposition~\ref{ratasym}, the projection $\pi(\cC)$ can either be closed or not. The next proposition shows that if $\ell$ an asymptote of $\cC$, then $\pi(\ell)\subseteq \overline{\pi(\cC)}$. Therefore, if $\ell$ is not a component of $\cC$, for instance, if $\cC$ is irreducible and not equal to $\ell$, then $\pi(\cC)$ is not closed.

\begin{pro}\label{asymnotcontained}
    Let $\cC\subset\R^2$ be a real algebraic curve, and $\ell$ be an asymptote of $\cC$. Then, $\pi(\ell)$ is contained in $\overline{\pi(\cC)}$. In particular, if $\ell$ is not a component of $\cC$, then $\pi(\cC)$ is not closed.
\end{pro}
\begin{proof}
     Let $(x,y)$ be a point of $\pi(\ell)$. Then, since $\ell$ is an asymptote of $\cC$, for each $\varepsilon>0$ there exist a point $P\in \cC$ and integer numbers $m,n$ such that $d(P,(x+n,y+m))<\varepsilon$. Then $\pi(\ell)\subset \overline{\pi(C)}$. Since $\ell$ is not a component of $\cC$, the intersections of $\pi(C)$ and $\pi(\ell)$ are countable and therefore $\pi(\ell)\not\subset\pi(C)$. It follows that $\pi(C)\neq \overline{\pi(C)}$, that is, $\pi(C)$ is not closed. 
\end{proof}

From Propositions~\ref{ratasym} and~\ref{asymnotcontained}, we get the following immediate corollary.

\begin{cor}
    Let $\cC\subset\R^2$ be an irreducible real algebraic curve such that all infinity branches have an asymptote $\ell$ with rational slope. Then, $\pi(\cC)$ is neither closed nor dense in the real torus.
\end{cor}

\begin{oss}\label{remhypandasym}
    The hypothesis of irreducibility in Theorem~\ref{compactproj} cannot be removed. Let $\cC$ be the plane curve defined by the equation 
    $$xy(xy-1)=0,$$ 
    i.e., the union of a hyperbola and its asymptotes. The projection $\pi(\cC)$ is closed and then compact in the real torus. Let $\widetilde{P}\notin \pi(\cC)$ be a point on the real torus. We show that there exists an open neighborhood of $\widetilde{P}$ disjoint from $\pi(\cC)$. Let $P=(x_P,y_P)\in\pi^{-1}(\widetilde{P})$. There exists $R\in\R_{>0}$ such that for any $(x,y)\in \cC(\R)\setminus [-R,R]^2$, we have $|x|<\frac{|x_P|}{2}$ or $|y|<\frac{|y_P|}{2}$. 
    
    Let $\mathcal{A}$ be the intersection of the orbit of $P$ under the action of $\Z^2$ with $[-R, R]^2$. Since $\mathcal{A}$ is finite, there exists $\delta>0$ such that for any $Q\in \mathcal{A}$ the open ball $B(Q,\delta)$ with center $Q$ and radius $\delta$ is disjoint from $\cC$. Taking $\delta':=\min\{\delta,\frac{|x_P|}{2},\frac{|y_P|}{2}\}$ we obtain that the orbit of $B(P,\delta')$ under the action of $\Z^2$ is disjoint from $\cC$, then its projection on the real torus is an open neighborhood of $\widetilde{P}$ disjoint from $\pi(\cC)$. 
\end{oss}

We end the paragraph by proving the following result on curves with an infinity branch that does not have asymptotes. 

\begin{pro}\label{noasym}
    Let $C$ be an irreducible non-compact real plane curve with an infinity branch $B$ whose leaves are not asymptotic to lines. Then $\pi(\cC)$ is dense in the real torus. 
\end{pro}

Our approach relies on Weyl’s theorem on polynomials~\cite{Weyl1916}. We first recall the necessary definitions.

\begin{defi}
    The real numbers $\alpha_1, \ldots, \alpha_s \in \mathbb{R}$ are said to be rationally independent if for each $s$-tuple of rational numbers $r_1, \ldots, r_s$ the condition
    $$r_1 \alpha_1 + \cdots + r_s \alpha_s = 0$$
    implies $r_1 = \cdots = r_s = 0$.
\end{defi}

\begin{defi}
    A sequence $x_n = (x_n^1, \ldots, x_n^k) \in \mathbb{R}^k$ is said to be \emph{uniformly distributed modulo 1} if, for each choice of $k$ intervals $[a_1, b_1], \ldots, [a_k, b_k] \subset [0, 1)$, we have
    $$ \frac{1}{n} \sum_{j=0}^{n-1} \prod_{i=1}^{k} \chi_{[a_i, b_i]}(\{x_j^i\}) \to \prod_{i=1}^{k} (b_i - a_i) \quad \text{as} \quad n \to \infty, $$
    where $\chi_{[a,b]}\colon \R\rightarrow \R$ is the characteristic function of the interval $[a,b]$ for all $a,b\in\R$.
\end{defi}

\begin{oss}
    It follows from the definition that if a sequence is uniformly distributed modulo 1 in $\R^k$ then it is dense in $\R^k/\Z^k$.
\end{oss}

\begin{thm}[Weyl]\label{weylpolgen}
    Let $p_i(n) = \alpha_d^{(i)} n^d + \alpha_{d-1}^{(i)} n^{d-1} + \cdots + \alpha_1^{(i)} n + \alpha_0^{(i)}$ be a polynomial with real coefficients. If there exists $1 \leq j \leq d$ such that $\alpha_j$, $\beta_j$ and 1 are rationally independent, then the set $S:=\{(p_1(n), \ldots,p_k(n))\mid n\in\N\} \subset \mathbb{R}^k$ is uniformly distributed modulo 1.
\end{thm}

\begin{proof}[Proof of Proposition~\ref{noasym}]
    Let $L$ be a leaf of $B$ parametrized by $(z,r(z))$ where $r\in\mathrm{Puis}_\C(z)$ have index of ramification $n$ and $|z|>M, M\in\R^+$. By Lemma~\ref{convergent} there exist a polynomial $p(x)$ such that the leaves $L$ and 
    $$\overline{L}:=\left\{(z,p(z^{1/n})) \bigm| |z|>\overline{M}\right\}$$
    converges. For each $\varepsilon>0$ there exist $z>0$ such that $d((z,r(z),(z,p(z^{1/n})))<\varepsilon$, then it is enough to show that $\overline{L}$ is dense in the real torus. We set $t:=z^{1/n}$ and we prove that 
    $$\{(t^n,p(t))\bigm| |t|>\overline{M}^{1/n}\}$$
    is dense in the real torus. We fix $y\in [0,1)$ and we consider the polynomials $(m+y)^n,p(m+y)$, with $m\in\N$. By assumption, $p$ has a term of degree $0<d\neq n$, otherwise $\overline{L}$ would be a line, so we can choose $y$ such that $(m+y)^n,p(m+y)$ have a pair of corresponding coefficients which are not rationally dependent with 1. Then, we can apply Theorem~\ref{weylpolgen}, which states that the set 
    $$\{((m+y)^n,p(m+y))\mid n\in\N\}$$ is uniformly distributed modulo 1, the the leaf $\overline{L}$ is dense in the real torus. 
\end{proof}

\subsection{Intersections of pairs of curves on the real torus}\label{secpairs}

We end this section discussing what happens if we intersect the projections of two algebraic curves on the real torus.

\begin{thm}\label{compactintersections}
    If the projections on the real torus of two real irreducible algebraic plane curves are both compact, then they have only finitely many intersections.   
\end{thm}
\begin{proof}
    From Theorem~\ref{compactproj}, we have to deal with three different cases, namely two compact curves, one compact curve and a line with rational slope, and two lines with rational slope.
    \begin{compactenum}[1.]
            \item \emph{$\cC_1,\cC_2$ are two irreducible compact curves.} Since $\cC_1,\cC_2$ are compact subsets of $\R^2$, there exist        
            $a_1,b_1,c_1,d_1,a_2,$ $b_2,c_2,d_2\in\R$ such that 
            $$\cC_i\subset [a_i,b_i]\times[c_i,d_i]\quad\text{for }i=1,2.$$
            Let $P=(x,y)\in \cC_1$ such that there exists an integer vector $(v_1,v_2)\in\Z^2$ satisfying $P+(v_1,v_2)\in \cC_2$. Then, 
            $$a_2\leq x+v_1\leq b_2\quad \text{and}\quad c_2\leq y+v_2\leq d_2,$$
            which gives
            $$a_2-b_1\leq v_1\leq b_2-a_2\quad \text{and}\quad c_2-d_1\leq v_2\leq d_2-c_1.$$
            Therefore, there are only finitely many choices for $(v_1,v_2)$. From Bezout's Theorem, if two algebraic curves have finitely many intersections, then there are only finitely many intersections on the torus.
        \item \emph{$\cC$ is an irreducible compact curve and $\ell$ is a line with rational      slope.} We can suppose that $(0,0)\in\ell$, that is, there exist coprime             integers $m,n$ such that $\ell:=\frac{m}{n}x$. From Bezout's identity, there exist $a, b\in\Z$ such that $am-bn=1$. We consider the linear change of             coordinates 
            $$T\colon \R^2\rightarrow \R^2\colon \begin{pmatrix}n\\m\end{pmatrix}\mapsto\begin{pmatrix}1\\0\end{pmatrix},\begin{pmatrix}a\\b\end{pmatrix}\mapsto\begin{pmatrix}0\\1\end{pmatrix}.$$
            We observe that $T\in\mathrm{SL}_2(\Z)$, then it induces a lattice automorphism of $\Z^2$ and so it is compatible with the action of $\Z^2$ on the torus. The line $T(\ell)$ is given by the equation 
            $$0=(m\quad -n)T^{-1}\begin{pmatrix}x\\y\end{pmatrix}=(m\quad -n)\begin{pmatrix}n & a\\m & b\end{pmatrix}\begin{pmatrix}x\\y\end{pmatrix}=(0\quad 1)\begin{pmatrix}x\\y\end{pmatrix}=y.$$
            Then the images of $T(\ell)$ under the action of $\Z^2$ are the lines of the form $\ell_k:=y=k$ with $k\in\Z$. Moreover, $T(\cC)$ is compact since $T$ is continuous, and then it is bounded. Therefore, there exist only finitely many $k$ such that $\ell_k$ and $T(\cC)$ have a non-empty intersection, then from Bezout's Theorem, there are only finitely many intersections on the torus.
            \item \emph{$\ell_1,\ell_2$ are two lines with rational slope.} 
            We can suppose that $(0,0)\in\ell_1\cap \ell_2$ and, as in the previous case, we can suppose that $\ell_1$ is described by the equation $y=0$ after a linear change of coordinates. Let $\ell_2: ax+by=0$ with $a,b\in\Z$. If $a=0$ then $\ell_1,\ell_2$ belongs to the same $\Z^2$-orbit. Otherwise, the intersections of the $\Z^2$-orbits of $\ell_1,\ell_2$ are described by the system of equations
            $$\begin{cases}
                y=0\\
                a(x+n)+b(y+m)=0.
            \end{cases}$$
            We get $x=-\frac{an+bm}{a}$, then $x$ can only assume values in the finite set $\{0,\frac{1}{a},\ldots\frac{a-1}{a}\}$. Therefore, $\ell_1,\ell_2$ have only finitely many intersections on the real torus.\qedhere
    \end{compactenum}
\end{proof}

\begin{oss}
    Theorem~\ref{compactintersections} does not follow only from the topological property of being compact. Indeed, if the curves have both compact projections but one is reducible, there could be infinitely many intersections. For instance, we can take the curve $\cC$ from Remark~\ref {remhypandasym} and a line with rational slope. The same conclusion holds if one is not algebraic. For instance, one can consider the closure $\overline{\cD}$ of the curve $$\cD=\left\{((1+e^{-t})\cos(t),(1+e^{-t})\sin(t))\mid t\in[0,+\infty)\right\}\subset \R^2. $$   
    The projection of $\overline\cD$ on the real torus is compact and intersects the projection of the line $x=0$ infinitely many times.
\end{oss}

\begin{oss}\label{linepower}
The converse of Theorem~\ref{compactintersections} does not hold. Define $\cC:y=x^k$, which has a dense image on the real torus, and $\ell:x=0$. Then the intersections of $\ell$ and $\cC$ on the real torus are projections of points given by the equations
$$\begin{cases}
    x=n,\\
    y=x^k,
\end{cases}$$
where $n\in\Z$. The solutions are $(n,n^k)$ with $n\in\Z$. Since these points belong to the same orbit for the action of $\Z^2$, there is only one intersection on the real torus.
\end{oss}

In view of the previous remark, we can ask the following question.

\begin{que}\label{Pairs}
    Which pairs of real irreducible algebraic curves have only finitely many intersections on the real torus?   
\end{que}

Given a real plane curve $\cC$ and an integer vector $v\in \Z^2$, we denote by $T_v(\cC)\subset \R^2$ the curve obtained by translating $\cC$ by the vector $v$. As we have seen in the proof of Theorem~\ref{compactintersections}, if $\cC_1, \cC_2\subset \A^2(\R)$ are two compact algebraic curves, then they have only finitely many intersections on the real torus since there are only finitely many integral vectors $v$ such that $T_v(\cC_1)$ intersect $\cC_2$. This can occur even if the curves are not compact, as shown in the following proposition. 

\begin{pro}
    Let $\ell: ax+by=0$ with $(a,b)\in\Z^2\setminus\{(0,0)\}$. Let $\cC$ be a real plane curve such that there exist two constants $c_1,c_2\in \R$ satisfying, for each $P=(x_P,y_P)\in \cC$,
    $$c_1\leq ax_P+by_P\leq c_2.$$
    Then, the curve $\cC$ has only finitely many intersections with the family of curves $T_v(\ell)$ as $v$ varies through all the integer vectors. Hence, $\ell$ and $\cC$ have only finitely many intersections on the real torus.
\end{pro}
\begin{proof}
    As in Theorem~\ref{compactintersections}, after a change of coordinates that induces an isomorphism on $\Z^2$, we can suppose that $\ell$ is the line $y=0$ and the points of $C$ satisfy the condition $d_1\leq y\leq d_2$. We have $$\left\{T_v(\ell)\mid v\in\Z^2\right\}=\left\{\ell_m:y=m\mid m\in\Z\right\},$$
    then the statement follows since there are only finitely many integers $m$ satisfying the condition $d_1\leq m\leq d_2$. 
\end{proof}

However, in Remark~\ref{linepower} the reason behind finiteness is different. Indeed, there are infinitely many integral vectors $v$ such that $\tau_v(\ell)$ intersects $\cC$, but all the intersections lie on the same orbit for the action of $\Z^2$ over $\R^2$. In Proposition~\ref{pairhyp}, we show that if one of the two curves is hyperbolic, this behavior cannot occur. 

\begin{defi}\label{defhyp}
    Let $\cC$ be an algebraic curve. Let $g$ be the genus, and $d$ be the number of points at infinity of a smooth completion of $\cC$. We say that $\cC$ is \emph{hyperbolic} if $2g-2+d$ is positive. 
\end{defi}

\begin{thm}[Siegel's Theorem,{\cite[Theorems 2.4,5.1]{Lang83}}]\label{SiegelThm}
    Let $K$ be a finitely generated field over $\Q$ and $R$ be a subring of $K$ finitely generated over $\Z$. Let $\cC$ be a hyperbolic algebraic curve defined over $K$. Then, there are only finitely many $R$-points on $\cC$. 
\end{thm}

\begin{pro}\label{pairhyp}
    Let $\cC_1, \cC_2$ be two real plane curves. Suppose that $\cC_1$ is hyperbolic, and that for infinitely many points $P\in\R^2$ there exists $v\in\Z^2$ such that $P\in \cC_1\cap \tau_v(\cC_2)$. Then $\cC_1$ and $\cC_2$ have infinitely many intersections on the real torus.
\end{pro}
\begin{proof}
     Suppose by contradiction that they have only finitely many intersections on the real torus. Then, there is a point $P\in\R^2$ and infinitely many integer vectors $v_i$, $i\in I$ such that $P+v_i\in \cC_1$. The points $P+v_i$ lie on a finitely generated ring over $\Q$. Therefore, Siegel's theorem implies that they cannot all lie on $\cC_1$. This leads us to a contradiction.
\end{proof}

In the following example, we show that the hypothesis in the previous proposition cannot be removed. More precisely, it is not sufficient to ask that the two curves are non-compact and one is hyperbolic.

\begin{es}
    Let $\cC$ be the plane curve defined by $y(y^2+3x^2)=4$. The curve $\cC$ has genus one since its completion in $\pr^2$ is a smooth cubic. From the defining equation, we get $0<y\leq 2$. Then, the projection on the real torus of the curve $\cC$ and the line defined by $y=0$ intersect only finitely many times. 
\end{es}

\section{Three lines}\label{sectionlines}
In this section, we address the instance of Question~\ref{QuestionTorus} involving three lines. We are given three affine lines $\ell_1,\ell_2,\ell_3\subset \R^2$, and we ask whether there are infinitely many triple intersections modulo the action of $\Z^2$ on $\R^2$. We will prove that under some assumptions, there are infinitely many points of triple intersection on the torus that are generated by the action of some automorphism of the plane which preserves the three lines (Proposition~\ref{infinitelines}), namely homotheties. Moreover,  we will prove that under stronger assumptions, there are infinitely many triple intersections modulo the action of the semigroup generated by integral translations and homotheties with integral ratio (Theorem~\ref{infSaction}). 
To prove Theorem~\ref{infinitelines}, we will suppose that the three lines have a common point, and we can make this assumption without losing generality if the three lines intersect on the torus. Moreover, we will suppose that they are not pairwise parallel. We start by analyzing separately the configurations for which these two assumptions do not hold. We denote by $O$ the point $(0,0)\in\R^2$.

\subsection{Degenerate configurations}

Let $\ell_1,\ell_2,\ell_3$ be three affine lines. Suppose that $\ell_1,\ell_2$ are parallel. Then, without loss of generality, we can suppose that there exist $\theta, \alpha \in \R$ such that
$$\ell_1:y=\theta x\quad \ell_2:y=\theta x+\alpha.$$
Moreover, we can translate the configuration to have $O\in \ell_1$. Then, the two lines meet on the torus if and only if there exist integers $m,n\in \Z$ such that 
\begin{equation*}
    \begin{cases}
        y=\theta x\\
        y=\theta (x - n)+\alpha -m,\\
    \end{cases}
\end{equation*}
has a solution. This lead to the condition $\alpha=\theta n+m$, that is $\alpha\in \langle 1,\theta\rangle_\Z$. In this case, the two lines coincide on the torus. 

If the line $\ell_3$ is parallel to $\ell_1$ and $\ell_2$, i.e., it is defined by the equation $y=\theta x+\beta$, then we have infinitely many triple intersections on the torus if and only if $\beta\in \langle 1,\theta\rangle_\Z$ and in this case they coincide on the torus. We can summarize as follows.

\begin{pro}
Let $\ell_1,\ell_2,\ell_3\subset \R^2$ be three parallel lines. If they are parallel to the $y$ axis, i.e. there exist $\alpha_1,\alpha_2,\alpha_3\in\R$ such that $\ell_i:y=\alpha_i$ for $i=1,2,3$, then they coincide on the torus if and only if $\alpha_1,\alpha_2,\alpha_3$ have the same fractional part. Otherwise, there exist $\theta,\beta_1,\beta_2,\beta_3\in\R$ such that 
$\ell_i:y=\theta x+\beta_i$ for $i=1,2,3$ and the three lines coincide on the torus if and only if
$$\beta_i-\beta_1\in \langle 1, \theta\rangle_\Z\quad\text{for }i=2,3.$$
\end{pro}

Now we suppose that $\ell_1$ and $\ell_2$ coincide on the torus and $\ell_1$, $\ell_3$ are not parallel. We can suppose that $\ell_1$ and $\ell_3$ meet in $O$. If at least one of them has an irrational slope, for example $\ell_1$, then there are infinitely many intersections on the torus. Indeed, for every translation of $r_3$ by an integer vector $(m,n)\in\Z^2$ which is not parallel to $r_3$, we get a different intersection point on $r_1$, and they are also different on the torus for Proposition~\ref{denselines}.
 
On the other hand, if both have rational slope, then they have only finitely many intersections on the real torus. Therefore, we have proved the following proposition.

\begin{pro}
    Let $\ell_1,\ell_2,\ell_3$ be three lines defined by $\ell_i: \alpha_i x+\beta_iy+\gamma_i=0$, with $(\alpha_1,\beta_1)$ and $(\alpha_2,\beta_2)$ linearly dependent. Then there are infinitely many triple intersections on the torus if and only if 
    $$\gamma_1-\gamma_2\in \langle 1,\frac{\alpha_1}{\beta_1}\rangle_\Z$$
    and at least one between $\frac{\alpha_1}{\beta_1}$ and $\frac{\alpha_3}{\beta_3}$ is irrational.
\end{pro}

Now we consider three lines $\ell_1,\ell_2,\ell_3$ which are not pairwise parallel. Then, we can suppose without loss of generality
$$\ell_1: y=\theta_1 x\quad \ell_2:y=\theta_2x+\tau\quad \ell_3:\alpha x+\beta y+\gamma=0, $$
with $\theta_1,\theta_2,\tau,\alpha,\beta,\gamma\in\R$. They meet on the torus if and only if
$$\begin{cases}
    y=\theta_1 x\\
    y=\theta_2(x+m_1)+\tau +n_1\\
    \alpha (x+m_2)+\beta (y+n_2)+\gamma=0
\end{cases}$$
has a solution for some $n_1,m_1,n_2,m_2$. This is equivalent to the equation 
$$\alpha\frac{\theta_2m_1+\tau+n_1}{\theta_2-\theta_1}+\beta\theta_1\frac{\theta_2m_1+\tau+n_1}{\theta_2-\theta_1}+\beta n_2+\gamma=0,$$
which means that the coefficients of the three lines satisfy an algebraic relation over $\Q$.

\subsection{Main case}

In this paragraph, we classify the triples of lines with infinitely many triple intersections on the real torus. The main idea is to start from two triple intersections and generate infinitely many others through homotheties centred at one of them.

\begin{thm}\label{Homotheties}
    Let $\ell_1, \ell_2,$ and $\ell_3$ be three affine lines with two distinct triple intersections on the torus. Suppose that one line has an irrational slope. Then, 
    there are infinitely many distinct triple intersections on the torus.
\end{thm}
\begin{proof}
    We can assume, up to translations and up to a map induced by a matrix in $\mathrm{SL}_2(\Z)$, that $\{O\}= \ell_1\cap \ell_2\cap \ell_3$ and that they are all distinct from the line $x=0$. 

    Under this assumptions, there exist three distinct real numbers $\theta_1,\theta_2,\theta_3$ such that 
    $$\ell_i=\left\{t(1,\theta_i)\mid t\in\R\right\}\quad i=1,2,3.$$
    We ask whether there exist infinitely many triples $(t_1,t_2,t_3)$ such that for $i,j\in\{1,2,3\}, i\neq j$ the following properties hold:
    \begin{align}
    t_i-t_j\in\Z\label{prop1}\\
    t_i\theta_i-t_j\theta_j\in\Z\label{prop2}.
    \end{align}
    We assume, without loss of generality, that $\ell_3$ has an irrational slope. Then distinct points on $\ell_3$ correspond to distinct points on the torus. Let $T=(t,t\theta_3)$ be a triple intersection different from $O$. Then there exist $m_1,n_1,m_2,n_2$ such that 
    $$T_1=(t+m_1,t\theta_3+n_1)\in \ell_1\setminus \{O\}\quad \text{and}\quad T_2=(t+m_2,t\theta_3+n_2)\in \ell_2\setminus\{O\}.$$ 
    Let $k\neq 0$ be an integer, and consider the homothety $\Phi_k$ with center $(0,0)$ and ratio $k$. Then $\Phi_k(T_1)\in \ell_1$ and $\Phi_k(T_2)\in \ell_2$. Moreover $$\Phi_k(T_1)-\Phi_k(T)=k(m_1,n_1)\in\Z^2\quad \text{and}\quad  \Phi_k(T_2)-\Phi_k(T)=k(m_2,n_2)\in\Z^2.$$
    The set $$\big\{\Phi_k(T)\mid k\in \Z\setminus\{0\}\big\}\subset \ell_3$$
    is an infinite set of points in $\R^2$ which correspond to distinct triple intersections on the torus.
\end{proof}

From Theorem~\ref{Homotheties}, we can construct examples of triples of lines with infinite triple intersections on the torus. Indeed, we can choose a line $\ell$ with irrational slope, a point $P$ on $\ell$ different from $O$ and define the points 
$$P_1=P+(m_1,n_1), \quad P_2=P+(m_2,n_2)\quad  \text{with } (m_1,n_1),(m_2,n_2)\in\Z^2.$$
Then the lines $\ell_1$ through $O$ and $P_1$ and $\ell_2$ through $O$ and $P_2$ have infinite triple intersections on the real torus.  

As a consequence of Theorem~\ref{Homotheties}, given three affine lines such that one of them has irrational slope, asking whether there are infinitely many points of triple intersection on the torus is equivalent to asking whether there are two different points of triple intersection on the torus. The complete answer is given in the following proposition. We make the same assumptions as in the proof of Theorem~\ref{Homotheties}.

\begin{pro}\label{infinitelines}
    Let $\theta_1, \theta_2, \theta_3$ be distinct real numbers. Then the affine real lines
    $$\ell_1:y=\theta_1x\quad \ell_2:y=\theta_2x\quad \ell_3:y=\theta_3x$$
    have infinitely many triple intersections on the real torus if and only if, except for reordering $\theta_1, \theta_2,$ and $\theta_3$, one of the following conditions holds:
    \begin{compactenum}[(a)]
        \item $\theta_1, \theta_2, \theta_3$ are irrational numbers and $\theta_1,\theta_2,\theta_3,\theta_1\theta_2,\theta_2\theta_3,\theta_3\theta_1$ satisfies the following relation: 
        \begin{align*}
        &a\theta_1+b\theta_2+c\theta_3+d\theta_1\theta_2+e\theta_2\theta_3+f\theta_3\theta_1=0\\
        \text{with }(a,b,c,d,&e,f)\in\Z^6\setminus\{(0,0,0,0,0,0)\},\text{ and } a+b+c=d+e+f=0.
        \end{align*}
        \item $\theta_1$ and $\theta_2$ are irrational numbers, $\theta_3$ is rational, and there exist a matrix  
        $$A:=\begin{pmatrix}a &b\\c & d\end{pmatrix}\in \mathrm{GL}_2(\Q)$$
        such that $\theta_1=A(\theta_2)$ and $\theta_3$ is fixed by the action of $A$.
    \end{compactenum}
\end{pro}
\begin{proof}
Suppose that there is a triple intersection different from $O$. By Properties~\eqref{prop1},\eqref{prop2} there exist \mbox{$t_1,t_2,t_3\in\R$,} $(m_1,m_2)\in\Z^2\setminus {(0,0)},$ and $(n_1,n_2)\in\Z^2\setminus {(0,0)}$ satisfying 
\begin{align*}
t_3=t_1+m_1,&\quad t_3=t_2+m_2,\\ 
n_1=t_1\theta_1-t_3\theta_3,&\quad n_2=t_2\theta_2-t_3\theta_3.
\end{align*}
Combining the equations, we get
$$n_1=(t_3-m_1)\theta_1-t_3\theta_3\quad\text{and}\quad n_2=(t_3-m_2)\theta_2-t_3\theta_3.$$
From the first one we get $t_3(\theta_1-\theta_3)=n_1+m_1\theta_1$ and since $\theta_1\neq \theta_3$ we can substitute the expression for $t$ in the second one, obtaining
$$n_2=\left(\frac{n_1+m_1\theta_1}{\theta_1-\theta_3}-m_2\right)\theta_2-\frac{n_1+m_1\theta_1}{\theta_1-\theta_3}\theta_3.$$
Rearranging the terms, we get
\begin{equation}\label{fund}
(\theta_1-\theta_3)(n_2+m_2\theta_2)=(\theta_2-\theta_3)(n_1+m_1\theta_1),
\end{equation}
which is equivalent to 
$$n_2\theta_1-n_1\theta_2+(n_1-n_2)\theta_3+(m_2-m_1)\theta_1\theta_2-m_2\theta_2\theta_3+m_1\theta_3\theta_1=0.$$

Moreover, we get the values
$$t_3=\frac{n_1+m_1\theta_1}{\theta_1-\theta_3}\quad t_1=\frac{n_1+m_1\theta_3}{\theta_1-\theta_3}\quad t_2=\frac{n_2+m_2\theta_3}{\theta_2-\theta_3}.$$ 
If $\theta_1,\theta_2,\theta_3$ are irrational numbers satisfying the relation stated in (a), then there exists a non-trivial integral solution $(m_1,n_1,m_2,n_2)\neq (0,0,0,0)$ to Equation~\eqref{fund}, and we can suppose $(m_1,n_1)\neq (0,0)$. From Equation~\eqref{fund} we get $(m_2,n_2)\neq (0,0)$. Therefore, $t_1,t_2,t_3\neq 0$, and there exists a triple intersection on the torus which is different from $O$. It follows from Theorem~\ref{Homotheties} that there are infinitely many of them.

Now, we suppose that $\theta_1,\theta_2$ are irrational and $\theta_3$ is rational. We can rewrite Equation~\eqref{fund} as
\begin{equation}\label{impo}
    \theta_1\big[\theta_2(m_2-m_1)+n_2+m_1\theta_3\big]=\theta_2(n_1+\theta_3m_2)+\theta_3(n_2-n_1).
\end{equation}
Then, if Equation~\eqref{impo} is non-trivial and it has a non-trivial solution, there exist $a,b,c,d\in\Q$ satisfying
$$\theta_1=\frac{a\theta_2+b}{c\theta_2+d}\quad \text{and}\quad\theta_3c+d=\frac{b}{\theta_3}+a.$$ 
We remark that $A=\big(\begin{smallmatrix}a &b\\c & d\end{smallmatrix}\big)$ is invertible since otherwise $\theta_1$ would be rational, and the second equation is equivalent to $c\theta_3^2+(d-a)\theta_3-b=0$, which expresses the fact that $\theta_3$ is a fixed point for $A$.

Now suppose that there exists $A$ as in statement (b). We have 
$$\begin{cases}
    m_2-m_1=c\\
    n_2+m_1\theta_3=d\\
    n_1+\theta_3m_2=a\\
    \theta_3(n_2-n_1)=b
\end{cases}$$
and we can suppose $\frac{b}{\theta_3},a,c,d\in\Z$. Then, choosing $m_1\in\Z$ such that $m_1\theta_3\in\Z$ we have the solution
$$m_2=c+m_1\quad n_2=d-m_1\theta_3\quad n_1=d-\frac{b}{\theta_3}-m_1\theta_3.$$
It gives 
$$t_1=\frac{d-\frac{b}{\theta_3}}{\theta_1-\theta_3}\quad\text{and}\quad t_2=\frac{\theta_3c+d}{\theta_2-\theta_3},$$
hence the choice of $m_1$ does not affect the values of $t_1,t_2$. Then there are infinitely many triple intersections if and only if $\theta_3c+d\neq 0$ or $d-\frac{b}{\theta_3}\neq 0$ since $a,b,c,d$ depend on a rational parameter. These conditions hold since $\theta_3\in\Q$ is fixed by $A$, and $\theta_3c+d = 0,d-\frac{b}{\theta_3}= 0$ are respectively equivalent to $A(\theta_3)=+\infty$ and $A^{-1}(\theta_3)=+\infty$.

If Equation~(\ref{impo}) is trivial we have 
$$m_2=m_1\quad n_2+m_1\theta_3=0\quad n_1+\theta_3m_2=0\quad n_1=n_2$$
and then $t_1=t_2=0$, so there is only one triple intersection.

If at least two between $\theta_1,\theta_2,$ and $\theta_3$ are rational, then the three lines have only finitely many triple intersections on the real torus from Theorem~\ref{compactintersections}.
\end{proof}

\begin{oss}
    The hypothesis of Proposition~\ref{Homotheties} cannot be removed. Indeed, three lines with rational slopes can have more than two triple intersections on the real torus, but only finitely many. 
\end{oss}

\begin{oss}
    From a geometric perspective, condition (a) of Proposition~\ref{infinitelines} implies that the point $(\theta_1,\theta_2,\theta_3)\in\A^3(\R)$ lies on a quadric $\mathcal{Q}$ in $\A^3$ defined over $\Q$ that contains the line $\ell :x=y=z$. However, not every point on $\mathcal{Q}$ corresponds to a triple of lines with infinitely many triple intersections on the real torus; for example, this fails for the points of $\ell$.
\end{oss}

\subsection{Action by homotheties}

In the previous paragraph, the presence of infinitely many orbits intersecting the three lines was explained by the action of another semigroup of endomorphisms of the plane, namely the semigroup of homotheties with center $O$ and integral ratio.  

Let $\Gamma$ be the set
$$\Gamma:=\big\{\tau_v\mid v\in\Z^2\big\}\cup \big\{\Phi_k\mid k\in\Z\setminus\{0\}\big\},$$
where $\tau_v$ is the translation of vector $v$ and $\Phi_k$ is the homothety of ratio $k$. Let $\mathcal{S}:=\langle \Gamma\rangle$ be the semigroup generated by the elements of $\Gamma$.  

In \cite[Proposition 4.1]{CZ2023}, is shown that, given three lines in $\pr^2$ and a projective automorphism $\beta$ of $\pr^2$, if there are infinitely many $\beta$-orbits of points in $\pr^2$ that intersect all the three lines, then there exists a projective automorphism $\gamma$ of $\pr^2$ such that these orbits are contained in finitely many orbits for the group generated by $\beta$ and $\gamma$. In contrast, under suitable hypotheses, we prove that the set of triple intersections of the three lines on the real torus consists of infinitely many different grand orbits for the action of $\mathcal{S}$, which are defined as follows.

\begin{defi}
    The \emph{grand orbit} of $x\in\R^2$ under the action of a semigroup $S$ is the set
    $$GO(x):=\left\{y\in\R^2 \mid \alpha(x)=\beta(y) \text{ for some } \alpha,\beta\in S\right\}.$$
\end{defi}

We make the same assumptions as in the proof of Theorem~\ref{Homotheties}.

\begin{thm}\label{infSaction}
    Let $\theta_1,\theta_2,\theta_3$ be real numbers and $\ell_1,\ell_2,\ell_3$ be affine real lines defined by $\ell_i:y=\theta_ix$. Then, there are infinitely many different grand orbits of points of $\R^2$ under the action of $\mathcal{S}$ which intersect $\ell_1, \ell_2$, and $\ell_3$  if and only if $\theta_1,\theta_2,\theta_3$ belongs to the same orbit under the action of $\mathrm{PGL}_2(\Q)$ over $\mathbb{P}^1(\R)$. In particular, $\ell_1,\ell_2,\ell_3$ satisfy the condition (a) of Theorem{\normalfont~\ref{infinitelines}}.  
\end{thm}

\begin{lemma}\label{lemmahom}
    Let $t,t'\in \R$ $m,n,k\in\Z$ such that 
    $$T_1=(t+m,t\theta_3+n)\in \ell_1\quad\text{and}\quad T_1'=(t'+km,t'\theta_3+kn)\in \ell_1.$$ 
    Then $T, T'$ belong to the same grand orbit under the action of $\mathcal{S}$.
\end{lemma}
\begin{proof}
    We have
    $$\begin{cases}
        \theta_1(t+m)=t\theta_3+n\\
        \theta_1(t'+km)=t'\theta_3+kn.
    \end{cases}$$
    Therefore,
    $$\begin{cases}
        (\theta_1-\theta_3)t=n-\theta_1m\\
        (\theta_1-\theta_3)t'=kn-k\theta_1m,
    \end{cases}$$
    so $t'=kt$ and $T, T'$ belong to the same grand orbit under the action of $\mathcal{S}$. 
\end{proof}

\begin{proof}[Proof of Theorem{\normalfont~\ref{infSaction}}]
    If there are different points under the action of $\mathcal{S}$, that is, there exist 
    \begin{align*}
        T_1=(t+m_1,t\theta_3+n_1)\in \ell_1,&\quad T_2=(t+m_2,t\theta_3+n_2)\in \ell_2,\\
        T_1'=(t'+m_1',t'\theta_3+n_1')\in \ell_1,&\quad T_2'=(t'+m_2',t'\theta_3+n_2')\in \ell_2
    \end{align*}
    such that 
    \begin{equation}\label{equhom}
        \frac{n_2+m_2\theta_2}{n_1+m_1\theta_1}=\frac{\theta_2-\theta_3}{\theta_1-\theta_3}=\frac{n_2'+m_2'\theta_2}{n_1'+m_1'\theta_1},
    \end{equation}
    then $\theta_1,\theta_2$ satisfy
    $$\frac{n_2+m_2\theta_2}{n_2'+m_2'\theta_2}=\frac{n_1+m_1\theta_1}{n_1'+m_1'\theta_1}$$
    and, as a consequence of Lemma~\ref{lemmahom}, the matrices 
    $$\begin{pmatrix}m_1 &n_1\\ m_1' &m_1'\end{pmatrix} \quad \text{and}\quad\begin{pmatrix}m_2 & n_2\\ m_2' &n_2'\end{pmatrix}$$
    are both invertible. Indeed, if there exist $p,q\in \Z\setminus\{0\}$ such that $p(m_1,n_1)=q(m_1',m_1')$, then $\Phi_p(T_1)=\Phi_q(T_1')$, that is, $T_1,T_1'$ belongs to the same grand orbit. 
    
    Then $\theta_1,\theta_2$ belong to the same orbit under the action of $\mathrm{PGL}_2(\Q)$, and by symmetry $\theta_3$ belongs to it, too. In particular, $\Q(\theta_1)=\Q(\theta_2)=\Q(\theta_3)$. Now we prove the converse statement.
    We have to prove that if $\theta_1,\theta_2,\theta_3$ belongs to the same orbit under the action of $\mathrm{PGL}_2(\Q)$ over $\mathbb{P}^1(\R)$, then there exist infinitely many 4-tuples $(m_1,n_1,m_2,n_2)$ pairwise linearly independent such that Equation~\eqref{equhom} holds. In other words, we suppose 
    $$\theta_1=\frac{a_1\theta_3+b_1}{c_1\theta_3+d_1}\quad\text{and} \quad\theta_2=\frac{a_2\theta_3+b_2}{c_2\theta_3+d_2}$$
    with $a_1,b_1,c_1,d_1,a_2,b_2,c_2,d_2\in\Z$ and we must find infinite 4-tuples not pairwise linearly dependent $(m_1,n_1,m_2,n_2)\in\Z^2$ satisfying
    $$\begin{cases}
        m_1a_1+n_1c_1=m_2a_2+n_2c_2\\
        m_1b_1+n_1d_1=m_2b_2+n_2d_2.
    \end{cases}$$
    The system has two equations and four variables, so its space of solution over $\Q$ has at least dimension two. Taking two independent solutions $w_1,w_2\in\Q^4$, we can clear denominators obtaining two integer solutions $\overline{w_1},\overline{w_2}\in\Z^4$. Then the set 
    $$\{\overline{w_1}+k\overline{w_2}\mid k\in\Z\}$$
    consists of infinitely many 4-tuples, not being pairwise linearly dependent.
\end{proof}

\section{Two lines and a curve}\label{Sectwolines}

We aim to classify all the triples $(\ell_1,\ell_2, \cC)$, where $\ell_1,\ell_2\subset \R^2$ are real lines and $\cC\subset \R^2$ is an irreducible real curve of degree at least 2, which have infinitely many triple intersections in the real torus. We will always suppose that $\ell_1,\ell_2$ do not belong to the same orbit for the action of $\Z^2$. Given a curve $\cC$, we define the following set:
\begin{align*}
    \mathcal{L}_\cC:=\{&(\ell_1,\ell_2) \mid \ell_1,\ell_2\text{ are real lines belonging to different $\Z^2$-orbits such that }\\
    &\cC,\ell_1,\ell_2 \text{ have infinitely many intersections on the real torus}\}.
\end{align*}

From Section~\ref{secproj}, we know that for any curve $\cC$, the set $\mathcal{L}_{\cC}$ does not contain pairs of lines with rational slope. Moreover, if $\cC$ is compact and $(\ell_1,\ell_2)$ belongs to $\mathcal{L}$, then $\ell_1,\ell_2$ have irrational slopes. 

In this section, we focus on determining whether the set $\mathcal{L}_{\cC}$ is empty for some families of curves. First, we prove that $\mathcal{L}_{\cC}$ is empty when $\cC$ is hyperbolic; then we construct examples showing that the degree alone does not obstruct infinitude; and finally, we address the case of conics. 

\begin{thm}\label{Twolinesandhyperbolic}
    Let $\ell_1,\ell_2\subset \R^2$ be two real lines that do not lie in the same $\Z^2$-orbit, and $\cC$ be an affine hyperbolic plane curve (see Definition~{\normalfont\ref{defhyp}}), all defined over a number field $K$. Then $\ell_1,\ell_2$ and $\cC$ have only finitely many points of triple intersection on the real torus. In particular, if $\cC$ is hyperbolic, then $\mathcal{L}_{\cC}$ is empty.
\end{thm}
\begin{proof}
    We observe that the intersection of two lines defined over the same field $K$ is a \mbox{$K$-rational} point, and in particular, there exists a finite set of places $S$ such that all the intersections between a line in the $\Z^2$-orbit of $\ell_1$ and a line in the $\Z^2$-orbit of $\ell_2$ are $S$-integral points. Therefore, a triple intersection is the image of an $S$-integral point of $\cC$ on the real torus. By applying Siegel's Theorem~\ref{SiegelThm} on the curve $\cC$, we conclude that there are only finitely many triple intersections on the real torus.
\end{proof}

From the previous theorem, we have that if $\mathcal{L}_\cC$ is nonempty, then $\cC$ has genus zero and its completion in $\pr^2$ has at most two points at infinity. With the following Examples~\ref{onepoint},~\ref{twopoints}, we show that for any $d\geq 2$ there are curves of degree $d$ with one point at infinity, and curves of degree $d$ with two points at infinity satisfying $\mathcal{L}_\cC\neq \emptyset$. Therefore, the degree of $\cC$ does not provide any obstruction.

\begin{es}[Curves with one point at infinity]\label{onepoint}
Let $d, k$ be positive integers satisfying $d\geq 2$ and $\sqrt[d]{k}\notin \Z$. We consider the curves
$$\ell_1: x=0\quad \ell_2:y=\sqrt[d]{k}x\quad \cC: x=y^{d},$$
and the set of points $P_n=(0,\sqrt[d]{k}n)$ with $n\in\Z$. We have 
$$P_n\in \ell_1,\quad P_n+(n,0)\in \ell_2\quad\text{and}\quad P_n+(kn^d,0)\in \cC.$$
Then $P_n$ is a point of triple intersection on the torus for each $n\in\Z$. These points are pairwise distinct on the torus since $\sqrt[d]{k}\notin \Z$. Therefore, $(\ell_1,\ell_2)\in \mathcal{L}_\mathcal{C}$.
\end{es}

\begin{oss}
   In the previous example, there are no other triple intersections. To determine them, we have to solve the following system
\begin{equation*}
    \begin{cases}
        x=n\\
        y=\sqrt[d]{k}x\\
        x+m=(y+h)^d
    \end{cases}
\end{equation*}
for each $n,m,h\in\Z$. This leads to the equation
$$n+m=(\sqrt[d]{k}n+h)^d=kn^d+d\sqrt[d]{k^{d-1}}n^{d-1}h+\sum_{i=0}^{d-2}\binom{d}{i}h^{d-i}\sqrt[d]{k^{i}}n^{i}.$$
Since $\sqrt[d]{k^{d-1}}\notin\Q$, we have $n^{d-1}h=0$. If $n=0$, we have $(x,y)=(0,0)$. If $h=0$ we have $m=kn^d-n$, which leads to the set of points $\{P_n\mid n\in\Z\}$. 

We observe $P_n=\Phi_n(0,\sqrt[d]{k})$, where $\Phi_n$ is the homothety of center $(0,0)$ and ratio $n$. Then, all the triple intersections belong to the $\mathcal{S}$-orbit of $P_1$.
\end{oss}

\begin{es}[Curves with two points at infinity]\label{twopoints}
    Let $\theta\in\R$ be a quadratic integer and $d>0$ be an integer. We consider the curves
    $$\ell_1:x=\theta y\quad \ell_2: x=-\theta y\quad \cC: (x+\theta y)^d(\theta y-x)=1.$$
    The intersections between the lines obtained by translating $\ell_1$ and $\ell_2$ by integral vectors are solutions of the following system:
    \begin{equation*}
        \begin{cases}
           x=\theta (y+m)+n\\ 
           x=-\theta (y+h)+k,
        \end{cases}
    \end{equation*}
    that is,
    \begin{equation*}
        \begin{cases} 
           x+\theta y=-\theta h+k\\
           \theta y-x=-\theta m-n
        \end{cases}
    \end{equation*}
    whose solutions are 
    $$(x,y)=\left(\frac{\theta(m-h)+n+k}{2},\frac{\theta(m+h)+k-n}{2\theta}\right).$$
    Substituting in the equation of $\cC$, we get 
    \begin{equation}\label{integers}
        (-\theta h+k)^d(-\theta m-n)=1.
    \end{equation}
    By Dirichlet's unit theorem, the group of units of the ring of integers of $\Q(\theta)$ has rank one. Then, there exist $a_1,b_1\in\Z$ such that the algebraic number 
    $$a_t+b_t\theta=(a_1+b_1\theta)^t$$
    is a unit for each $t\in\Z$.
    
    If we set $h=-b_t, k=a_t, m=-b_{-dt}=b_{dt}, n=-a_{-dt}=-a_{dt}$, we obtain a solution of Equation~\ref{integers}, which corresponds to the point 
    $$(x_t,y_t)=\left(\frac{\theta(b_{dt}+b_{t})+a_{t}-a_{dt}}{2},\frac{\theta(b_{dt}-b_{t})+a_t+a_{dt}}{2\theta}\right).$$
    The set $\{(x_t,y_t)\mid t\in\Z\}$ corresponds to a set of infinitely many distinct points on the real torus since the sequence $t\mapsto b_{dt}+b_{t}$ is strictly increasing. Therefore, $(\ell_1,\ell_2)\in\mathcal{L}_\mathcal{C}$.
\end{es}

\subsection{Two lines and a non-compact conic}

Now we prove that if $\cC$ is an irreducible parabola or an irreducible hyperbola defined over $\Q$, then $\mathcal{L}_\cC$ is not empty. In particular, if $\cC$ is an irreducible non-compact conic, then for all but at most two rational numbers $q$ (depending on $\cC$), $\mathcal{L}_\cC$ contains infinitely many pairs $(\ell_1,\ell_2)$ where $\ell_1$ has slope equal to $q$. We state the main result separately for parabolas and hyperbolas.

\begin{thm}\label{mainparabolas}
    Let $\cC$ be an irreducible parabola defined by the equation
    $$\cC:(ax+by)^2+cx+dy+e=0\quad a,b,c,d,e\in\Z.$$
    Let $\ell_1$ be the line defined by the equation
    $$\ell_1: ux+vy=0\quad u,v\in\Z.$$
    Suppose $(a:b)\neq (u:v)$ in $\pr^1(\Q)$. Then, there exist infinitely many real lines $\ell_2$ such that $\cC,\ell_1,\ell_2$ have infinitely many triple intersections on the real torus.
\end{thm}

\begin{thm}\label{mainhyperbolas}
    Let $\cC$ be an irreducible hyperbola defined by the equation
    $$\cC:ax^2+bxy+cy^2+dx+ey+f=0\quad a,b,c,d,e,f\in\Z.$$
    Let $\ell_1$ be the line defined by the equation
    $$\ell_1: ux+vy=0\quad u,v\in\Z.$$
    If $ax^2+bxy+cy^2=(q_1x+q_2y)(q_3x+q_4y)$ with $q_1,q_2,q_3,q_4\in \Q$, i.e., the points at infinity of $\cC$ are defined over $\Q$, suppose $(u:v)\in \pr^1(\Q)\setminus\{(q_1:q_2),(q_3:q_4)\}$. Then, there exist infinitely many real lines $\ell_2$ such that $\cC,\ell_1,\ell_2$ have infinitely many triple intersections on the real torus.
\end{thm}

First, we need a technical lemma that ensures the existence of suitable slopes for the lines. We deduce it from Hilbert's Irreducibility Theorem.

\begin{thm}[Hilbert's Irreducibility Theorem {\cite[Theorem 46]{Schinzel2000}}]\label{Hilbert}
    Let 
    $$f_i(X_1,\dots,X_r,Y_1,\dots,Y_s)\in \Q[X_1,\dots,X_r,Y_1,\dots,Y_s]$$
    be irreducible polynomials for $1\leq i\leq n$. There exist infinitely many $r$-tuple of integral numbers $(a_1,\dots,a_r)$ such that $f_i(a_1,\dots,a_r,Y_1,\dots,Y_s)$ is irreducible in $\Q[Y_1,\dots,Y_s]$ for \mbox{$1\leq i\leq n$.}
\end{thm}

\begin{lemma}\label{existence}
    Let 
    $$P(x,y)=ax^2+bxy+cy^2+dx+ey+f\in\Z[x,y]$$
    be a polynomial irreducible over $\overline{\Q}$. Suppose $a\neq 0$ and $b^2-4ac\geq 0$. Then, there exist infinitely many pairs $(\theta,n)$ with $\theta\in\R\setminus\Q$ and $n\in\Z$ such that $f(\theta,n)=0$. Moreover, if $b^2-4ac> 0$, we can choose $n$ such that 
    $$\frac{(bn+d)^2-4a(cn^2+en+f)}{b^2-4ac}$$
    is not the square of a rational number.
\end{lemma}
\begin{proof}
    Since $P(x,y)$ is irreducible, we have
    $$\det\begin{pmatrix}
        a & b/2 & d/2\\
        b/2 & c & e/2\\
        d/2 & e/2 & f
    \end{pmatrix}=-\frac{1}{4}(ae^2-ebd-4afc+cd^2+fb^2)\neq 0.$$
    Since $a\neq 0$, the condition is equivalent to 
    \begin{equation}\label{det}
        4a(ae^2-ebd-4afc+cd^2+fb^2)-b^2d^2+b^2d^2=(4ac-b^2)(4af-d^2)-(2ae-bd)^2\neq 0  
    \end{equation}
    Let $n$ be an integer. The equation
    $$P(x,n)=ax^2+bxn+cn^2+dx+en+f=0$$
    has two real, irrational solutions if and only if its discriminant
    $$\Delta_n := (bn+d)^2-4a(cn^2+en+f)=(b^2-4ac)n^2+(2bd-4ae)n+d^2-4af$$
    is a non-square positive integer. 

    Suppose $b^2-4ac=0$. Then, by Equation~\eqref{det}, $2bd-4ae\neq 0$. It follows that $(\Delta_n)_{n\in\Z}$ is an arithmetic progression, so there exist infinitely many $n$ such that $\Delta_{n}$ is a non-square positive integer. Indeed, there are only finitely many $m\in\Z$ such that $\Delta_{m}$ and $\Delta_{m+1}$ are both perfect squares. 
     
    Suppose $b^2-4ac>0$. Then $\Delta_n>0$ for all but finitely many $n$ since it is a polynomial of degree 2 in $n$ with positive leading coefficient. Let us consider the polynomial
    $$q(T)=T^2+\frac{2bd-4ae}{b^2-4ac}T+\frac{d^2-4af}{b^2-4ac}\in\Q[T].$$
    The discriminant of $q(T)$ as a polynomial in $T$ is 
    $$\frac{1}{(b^2-4ac)}\left[(2ae-bd)^2-(4ac-b^2)(4af-d^2)\right],$$
    so it is nonzero by Equation~\eqref{det}. 
    It follows that $q(T)$ is not the square of a polynomial of degree 1, and we can apply Hilbert's Irreducibility Theorem~\ref{Hilbert} to the polynomials 
    $$U^2-q(T), V^2-(b^2-4ac)q(T)\in \Q[U,V,T]. $$
    Since $\Delta_n=(b^2-4ac)q(n)$, there exist infinitely many $n$ such that $q(n)$ and $\Delta_{n}$ are not perfect squares and $\Delta_{n}$ is positive.
\end{proof}

To show the existence of infinitely many triple intersections on the real torus, we will use the following theorem, which characterizes affine curves that contain infinitely many integral points.

\begin{thm}[{\cite[Theorem 1.2]{ABP2009}}]\label{ABPthm}
    Let $\cC\subset \A^n$ be an affine curve defined over a number field $K$, and $O_K$ be the ring of integers of $K$. Let $\widetilde{\cC}$ be a projective completion of the normalization $\cC^{\nu}$ of $\cC$ and $\cC_\infty=\widetilde{\cC}\setminus \cC$. The following two conditions are equivalent:
    \begin{compactenum}[(a)]
        \item The set $\cC(\cO_K)$ is infinite.
        \item The curve $\cC$ is of genus 0, the set $\cC(\cO_K)$ contains a non-singular point, and the set $\cC_{\infty}$ has one of the following properties:
        \begin{compactenum}[(i)]
            \item $|\cC_{\infty}| = 1$,
            \item $|\cC_{\infty}| = 2$, both points at infinity are defined over $K$, and $K$ is neither $ \Q$ nor an imaginary quadratic field,
            \item $|\cC_{\infty}| = 2$, the points at infinity are conjugate over $K$, and the field of definition of the points at infinity is not a CM-extension of $K$.
        \end{compactenum}
    \end{compactenum}
\end{thm}

Now we prove the central result on non-compact irreducible conics, which forms the core of the proofs of Theorems~\ref{mainparabolas},~\ref{mainhyperbolas}. The general results will later follow from a reduction argument.

\begin{pro}\label{Mainconics}
    Let $\cC\subset\R^2$ be a non-compact irreducible conic defined by the equation
    $$\cC:ax^2+bxy+cy^2+dx+ey+f=0\quad \text{ with }\quad a,b,c,d,e,f\in\Z, a\neq 0.$$
    There exist infinitely many pairs $(\theta, q)$ with $\theta\in\R\setminus\Q$ and $q\in\Q$ such that the curves $\cC$, $\ell_1:y=0$, $\ell_2: x=\theta y+q$ have infinitely many triple intersections on the real torus. In particular, $\mathcal{L}_\cC$ is not empty.
\end{pro}
\begin{proof}
        Let $\theta\in\R\setminus\Q$, $n_0\in\Z$ as in Lemma~\ref{existence}. We have to prove that, for some $q\in\Q$, the system 
    \begin{equation*}\label{system}
        \begin{cases}
            y=n\\
            x=\theta (y +  k)+ q+m\\
            ax^2+bxy+cy^2+dx+ey+f=0
        \end{cases}
    \end{equation*}
    has solutions for infinitely many integral triples $(n,m,k)$. By substitution, we get
    \begin{align*}
        0=&\;a[\theta(n+k)+m+q]^2+b[\theta(n+k)+m+q]n+cn^2+d[\theta(n+k)+m+q]+en+f\\
        = &\;a\theta^2(n+k)^2+2a(m+q)(n+k)\theta+a(m+q)^2+bn(n+k)\theta\\
        &+b(m+q)n+cn^2+d(n+k)\theta+d(m+q)+en+f\\
        =&-(n+k)^2(b\theta n_0+cn_0^2+d\theta+en_0+f)+2a(m+q)(n+k)\theta+a(m+q)^2+bn(n+k)\theta\\
        &+b(m+q)n+cn^2+d(n+k)\theta+d(m+q)+en+f
    \end{align*}
    Since $\theta$ is irrational, the last equation leads to the following system, where we have defined
    $$\gamma:=cn_0^2+en_0+f\quad\text{and}\quad\beta:=bn_0+d:$$
    $$\begin{cases}
        \beta(n+k)^2=2a(m+q)(n+k)+(bn+d)(n+k)\\
        \gamma(n+k)^2=a(m+q)^2+(bn+d)(m+q)+cn^2+en+f
    \end{cases}$$
    Supposing $n + k\neq 0$, from the first equation we obtain
    $$m+q=\frac{\beta(n+k)-(bn+d)}{2a}.$$
    By substitution, we get
    \begin{equation}\label{newconic}
        \gamma(n+k)^2=a\left[\frac{\beta(n+k)-(bn+d)}{2a}\right]^2+(bn+d)\left[\frac{\beta(n+k)-(bn+d)}{2a}\right]+cn^2+en+f
    \end{equation}
    Equation~(\ref{newconic}) defines a conic $\cD$ over $\Q$ in the coordinates $(n+k,n)$, which has the integral point $(1, 1-n_0)$. The homogeneous part of the equation is
    $$\gamma(n+k)^2=a\left[\frac{\beta(n+k)-bn}{2a}\right]^2+bn\left[\frac{\beta(n+k)-bn}{2a}\right]+cn^2$$
    which leads to
    \begin{equation}\label{omogenea}
        (\beta^2-4a\gamma)(n+k)^2=(b^2-4ac)n^2.
    \end{equation}
    We observe $\Delta:=\beta^2-4a\gamma>0$. Indeed, $\Delta$ is the discriminant of a polynomial that has two real irrational roots. Since $C$ is non-compact, we have $b^2-4ac\geq 0$. Therefore, we can apply Theorem~\ref{ABPthm} on the conic $\mathcal{D}$ with $K=\Q$. Indeed, one of the following conditions holds:
    \begin{compactenum}[(i)]
        \item $b^2-4ac=0$, that is, $\cC$ is a parabola and $\cD$ has only one point at infinity;
        \item $b^2-4ac>0$ and $\frac{4a\gamma-\beta^2}{4ac-b^2}$ is not the square of a rational number by Lemma~\ref{existence}, that is, $\cC$ is a hyperbola and $\cD_\infty$ consists of two real irrational points.
    \end{compactenum}
    It follows that there exist infinitely many integral solutions to Equation~\eqref{newconic}. For infinitely many of them, there exists an integer $0\leq N< 2a$ such that the congruence
    $$\beta(n+k)-(bn+d)\equiv N\quad \text{mod }2a$$
    holds. We set $q:=\frac{N}{2a}$, obtaining that $m$ is an integer and hence the System~\eqref{system} has solutions for infinitely many integral values of $(k,m,n)$. 
\end{proof} 

\begin{proof}[Proof of Theorem~{\normalfont\ref{mainparabolas}}]
    If $u=0$, we can suppose $v=1$. By hypothesis $a\neq 0$, then, the statement follows from Lemma~\ref{existence} and Proposition~\ref{Mainconics}. If $v=0$, we can apply the same argument switching $x$ and $y$.
    
    Now suppose $u,v\neq 0$. Without loss of generality, we assume $\mathrm{gcd}(u,v)=1$. By Bezout's identity, there exist integers $m, n$ satisfying the relation $mv + nu = 1$. Consider the map 
    $$T\colon \R^2\rightarrow \R^2\colon\begin{pmatrix}
        x\\
        y
    \end{pmatrix}\mapsto 
    \begin{pmatrix}
        vx+ny\\
        -ux+my
    \end{pmatrix}$$
    The inverse images of $\ell_1$ and $\cC$ under the map $T$ are defined by the equations
    \begin{align*}
        \ell_1':0&=u(vx+ny)+v(-ux+my)=y\\
        \cC':0&=[a(vx+ny)+b(-ux+my)]^2+c(vx+ny)+d(-ux+my)+e=0.
    \end{align*}
    The map $T$ is described by a matrix in $\mathrm{SL}_2(\Z)$, then it is an isomorphism of $\R^2$ which preserves the lattice $\Z^2$, and so proving the theorem for the pair $\cC,\ell_1$ is equivalent to proving it for the pair $C',\ell'_1$. Since $(a:b)\neq (u:v)$, we have $av-bu\neq 0$ and the statement follows from Lemma~\ref{existence} and Proposition~\ref{Mainconics}. 
\end{proof}

\begin{proof}[Proof of Theorem~{\normalfont\ref{mainhyperbolas}}]
    As in the proof of Theorem~\ref{mainparabolas}, we can suppose $u,v\neq 0$ and $\mathrm{gcd}(u,v)=1$. We again consider the map 
    $$T\colon \R^2\rightarrow \R^2\colon\begin{pmatrix}
        x\\
        y
    \end{pmatrix}\mapsto 
    \begin{pmatrix}
        vx+ny\\
        -ux+my
    \end{pmatrix}$$
    The inverse images of $\ell_1$ and $\cC$ under the map $T$ are defined by the equations
    \begin{align*}
        \ell_1':0&=u(vx+ny)+v(-ux+my)=y\\
        \cC':0&=a(vx+ny)^2+b(vx+ny)(-ux+my)+c(-ux+my)^2\\
        &+d(vx+ny)+e(-ux+my)+f=0.
    \end{align*}
    The coefficient of $x^2$ in the equation of $\cC'$ is $av^2-buv+cu^2\neq 0$, since by hypothesis the point $(-v:u:0)\in\pr^2(\Q)$ is not a point at infinity of $\cC$. The statement follows from Lemma~\ref{existence} and Proposition~\ref{Mainconics} as in the proof of Theorem~\ref{mainparabolas}. 
\end{proof}

\begin{oss}
    The hypothesis $(a:b)\neq (u:v)$ in $\pr^1(\Q)$ in Theorem~\ref{mainparabolas} cannot be removed. Suppose $(a:b)= (u:v)$ in $\pr^1(\Q)$, and, without loss of generality, $u\neq 0$ and if $v\neq 0$ then $\gcd(u,v)=1$. Then, the parabola $\cC$ is given by the equation
    $$C:k^2(ux+vy)^2+cx+dy+e=0\quad k\in\Z.$$
    Translating the line $\ell$ by the integer vector $(m,n)\in\Z^2$, we obtain the line $\ell_{m,n}$ of equation
    \begin{equation}\label{rettaoss}
        \ell_{m,n}: x=\frac{1}{u}(-vy+um+vn).
    \end{equation}
    By substitution, we get
    $$k^2(um+vn)^2+\frac{c}{u}(-vy+um+vn)+dy+e=0,$$
    which gives, if $d-\frac{cv}{u}\neq 0$, 
    $$y=-\left(d-\frac{cv}{u}\right)^{-1}\left(k^2(um+vn)^2+\frac{c}{u}(um+vn)+e\right).$$
    Therefore, $y\in\Z[\frac{1}{ud-cv}]$, and so the intersection of $\cC$ with any line of the form $\ell_{m,n}$ belongs to the union of finitely many $\Z^2$-orbits. If $ud-cv=0$, then the equation for $\cC$ becomes
    \begin{equation}\label{parabolaoss}
        \cC:k^2(ux+vy)^2+h(ux+vy)+e=0\quad h,k\in\Z.
    \end{equation}
    If there is an intersection between $\ell_{m,n}$ for some $(m,n)\in\Z^2$ and $\cC$, then all the solutions $(x,y)$ of Equation~(\ref{rettaoss}) satisfies Equation~(\ref{parabolaoss}), hence $\cC$ is not irreducible. 
    Therefore, for any other line $\ell'$, the curves $\cC,\ell, \ell'$ have at most finitely many triple intersections on the real torus.

    The same remark holds for Theorem~\ref{mainhyperbolas}. Indeed, one can provide an analogous argument for a hyperbola with the points at infinity defined over $\Q$ and a line containing one of them.    
\end{oss}

\subsection{Two lines and a compact conic}
The only case left for conics is the compact case. We consider the unit circle $\cC$ defined by the equation $x^2+y^2=1$, and we ask if there exist real numbers $\theta_1,\theta_2$ such that there exist infinitely many orbits for the action of $\Z^2$ on $\R^2$ which intersect $\cC$ and the lines $$\ell_1:y=\theta_1x\quad\text{and}\quad \ell_2: y=\theta_2x.$$

The intersections of a line in the $\Z^2$-orbit of $\ell_1$ and one in the $\Z^2$-orbit of $\ell_1$ are the solutions of the following equations, where $m_1,m_2,n_1,n_2$ are integers:
\begin{equation*}
    \begin{cases}
        y+n_1=\theta_1(x+m_1)\\
        y+n_2=\theta_2(x+m_2).
    \end{cases}
\end{equation*}
Then, the set of intersections is 
$$\left\{\left(\frac{n_1-n_2+\theta_2m_2-\theta_1m_1}{\theta_1-\theta_2},\frac{\theta_2n_1-\theta_1n_2+\theta_1\theta_2(m_2-m_1)}{\theta_1-\theta_2}\right)\mid m_1,m_2,n_1,n_2\in\Z\right\}.$$
Therefore, our question is equivalent to asking if there exist $\theta_1,\theta_2\in\R$ such that the following equation admits infinitely many integer solutions $(m_1,n_1,m_2,n_2)$:
\begin{equation}\label{eqcircle}
    (n_1-n_2+\theta_2m_2-\theta_1m_1)^2+(\theta_2n_1-\theta_1n_2+\theta_1\theta_2(m_2-m_1))^2=(\theta_1-\theta_2)^2.
\end{equation} 

The following result excludes the possibility of replicating the argument given in Proposition~\ref{Mainconics} in the case of the unit circle.

\begin{thm}
    If $[\Q(\theta_1):\Q],[\Q(\theta_2):\Q]\leq 2$, then $\ell_1, \ell_2,$ and $\cC$ have only finitely many triple intersections on the real torus.
\end{thm}
\begin{proof}
    We observe that if $\theta_1$ is rational, then $\ell_1$ and $\cC$ have only finitely many triple intersections on the real torus, and the same holds for $\theta_2$. If $[\Q(\theta_1):\Q]=[\Q(\theta_2):\Q]=2$, then $K:=\Q(\theta_1,\theta_2)$ is either a quadratic or a biquadratic extension of $\Q$. Therefore, $K$ is a totally real extension of $\Q$. Since the points at infinity of the curve $\cC$ are defined over the field $K(i)$, which is a CM-field, then the set $\cC(\mathcal{O}_K)$ is finite by Theorem~\ref{ABPthm}. It follows that there are only finitely many solutions to Equation~\ref{eqcircle}.
\end{proof}

We make the following conjecture, and in the remaining part of this section, we give some evidence.

\begin{con}
    Let $\theta_1,\theta_2$ be real numbers. Then, $\ell_1:y=\theta_1x$, $\ell_2:y=\theta_2x$, and $\cC:x^2+y^2=1$ have only finitely many triple intersections on the real torus. 
\end{con}

Suppose $[\Q(\theta_1,\theta_2):\Q]=d$, and $\theta_1,\theta_2$ are both irrational. Let $u\in \Q(\theta_1,\theta_2)$ such that $1,u,u^2,\dots u^{d-1}$ is a basis of $\Q(\theta_1,\theta_2)$ over $\Q$ as a vector space. There exist $a_i,b_i,c_i\in\Q$ for $0\leq i\leq d_1$ satisfying
$$\theta_1=\sum_{i=0}^{d-1} a_{i}u^{i}\qquad \theta_2=\sum_{i=0}^{d-1} b_{i}u^{i}\qquad u^d=\sum_{i=0}^{d-1} c_iu^i.$$
We can rewrite Equation~\ref{eqcircle} as 
$$\sum_{j=0}^{d-1} g_j\left(\{a_i\}_{0\leq i\leq d-1},\{b_i\}_{0\leq i\leq d-1},\{c_i\}_{0\leq i\leq d-1},m_1,m_2,n_1,n_2\right)u^j=0$$
Since $u$ is irrational and $m_1,m_2,n_1,n_2,a_i,b_i,c_i$ for all $0\leq i\leq d-1$ are rational, a solution to Equation~\ref{eqcircle} corresponds to an integral point of the affine variety $V\subset \A^4$ defined by
    \begin{equation*}
        \begin{cases}
            g_0\left(\{a_i\}_{0\leq i\leq d-1},\{b_i\}_{0\leq i\leq d-1},\{c_i\}_{0\leq i\leq d-1},m_1,m_2,n_1,n_2\right)=0\\
            g_1\left(\{a_i\}_{0\leq i\leq d-1},\{b_i\}_{0\leq i\leq d-1},\{c_i\}_{0\leq i\leq d-1},m_1,m_2,n_1,n_2\right)=0\\
            \vdots\\
            g_{d-1}\left(\{a_i\}_{0\leq i\leq d-1},\{b_i\}_{0\leq i\leq d-1},\{c_i\}_{0\leq i\leq d-1},m_1,m_2,n_1,n_2\right)=0
        \end{cases}
    \end{equation*}
where $a_i,b_i,c_i$ are fixed. From Equation~\eqref{eqcircle}, it is easy to see that $V$ always contains the curve $\widetilde{\cC}\subset \A^4$ defined by 
$$n_1=n_2\qquad m_1=m_2\qquad n_1^2+m_1^2=1.$$

\begin{pro}\label{ifcurve}
    In the same notation, if $V$ is a curve, then $V(\Z)$ is finite. Hence, $\ell_1,\ell_2,$ and $\cC$ have only finitely many intersections on the real torus.
\end{pro}
\begin{proof}
    Suppose that an irreducible component $V'$ of $V$ contains infinitely many $\Z$-integral points. Then $V'$ is defined over $\Q$. If its points at infinity are all complex, then $V'$ is compact in $\A^4(\R)$, and so it has only finitely many $\Z$-integral points. Now, suppose that $P_0=(x_0:y_0:z_0:w_0:0)$ is a point at infinity of $V'$ defined over $\R$. The point $P_0$ is also a point at infinity of the quadric hypersurface defined by Equation~\ref{eqcircle}, and it then satisfies the equation
    $$(z_0-w_0+\theta_2y_0-\theta_1x_0)^2+(\theta_2z_0-\theta_1w_0+\theta_1\theta_2(y_0-x_0))^2=0.$$
    Since $P_0$ is defined over $\R$, we get
    \begin{equation*}
        \begin{cases}
            z_0-w_0+\theta_2y_0-\theta_1x_0=0\\
            \theta_2z_0-\theta_1w_0+\theta_1\theta_2(y_0-x_0)=0.
        \end{cases}
    \end{equation*}
    It follows $z_0=\theta_1x_0$ and $w_0=\theta_2y_0$, and so the point $P_0=(x_0:y_0:\theta_1x_0:\theta_2y_0:0)$ is defined over a field of degree at least 3 over $\Q$. Therefore, $V'$ has at least three points at infinity, and by Siegel's theorem, the set $V'(\Z)$ is finite.
\end{proof}

Now we exhibit a particular case where the hypothesis of Proposition~\ref{ifcurve} is satisfied. Denote by $K$ the field $\Q(\theta_1,\theta_2)$ and suppose $[K:\Q]=3$ and $\theta_1,\theta_2$ are both irrational. Therefore, we have $K=\Q(\theta_1)=\Q(\theta_2)$, and there exist rational numbers $a,b,c,d,e,f$ satisfying
$$\theta_1^3+a\theta_1^2+b\theta_1+c=0\quad\text{and}\quad \theta_2=d\theta_1^2+e\theta_1+f.$$
If we substitute the expression of $\theta_1,\theta_2$ in Equation~\eqref{eqcircle} we obtain an equation of the form
\begin{equation}\label{eqcircle2}
g_2(m_1,m_2,n_1,n_2)\theta_1^2+g_1(m_1,m_2,n_1,n_2)\theta_1+g_0(m_1,m_2,n_1,n_2)=0,
\end{equation}
    where $g_0,g_1,g_2$ are polynomials of degree 2 whose coefficients depend on $a,b,c,d,e,f$. A solution to Equation~\ref{eqcircle2} corresponds to a solution of the system
    \begin{equation*}
        \begin{cases}
            g_0(m_1,m_2,n_1,n_2)=0\\
            g_1(m_1,m_2,n_1,n_2)=0\\
            g_2(m_1,m_2,n_1,n_2)=0
        \end{cases}
    \end{equation*}
    which describes the intersection of three quadrics in $\A^4$ with $\Z^4$. We can show that the three quadrics are linearly independent for any choice of $a,b,c,d,e,f$. The explicit expressions for $g_0,g_1,g_2$ are the following:
\allowdisplaybreaks
\begin{align*}
g_0(m_1,m_2,n_1,n_2)&= m_{1}^{2}( a^{3} c d^{2}-2 a b c d^{2}+ c^{2} d^{2} + 2 b c d e + a c e^{2} -2 a^{2} c d e+ 2 a c d f-2 c e f) \\
&+ m_{2}^{2}( a^{3} c d^{2} -2 a b c d^{2} + a c d^{2}+ 2 a c d f -2 a^{2} c d e+ c^{2} d^{2}+ 2 b c d e + a c e^{2} -2c d e\\ 
&-2 c e f+f^2)+ 2 m_{1} m_{2}( -a^{3} c d^{2}+ 2a b c d^{2}+2a^{2} c d e - c^{2} d^{2}-2 b c d e -2 a c e^{2}\\
&- 2 a c d f+ 2 c d  + 4 c e f)+ n_{1}^{2}( a c d^{2} -2 c d e+f^2+1)+n_2^2+2 n_{1} n_{2} (c d-1)\\
&+ m_{1} n_{1} ( -4 a c d e -2b c d^{2}+2 a^{2} c d^{2}+ 2 c e^{2}+ 4 c d f)+ 2 m_{1} n_{2}( a c d - c e)\\
&+ 2 n_{1} m_{2}(2 a c d e -a^{2} c d^{2}+ b c d^{2}-c e^{2} -2 c d f+f) + 2 m_{2} n_{2}( c e-acd-f) \\
&- f^{2}  -  a c d^{2} -2 c d + 2 c d e \\
g_1(m_1,m_2,n_1,n_2)&= m_{1}^{2}( a^{3} b d^{2}-2 a b^{2} d^{2}- a^{2} c d^{2}-2a^{2} b d e + 2 b c d^{2}+ 2 b^{2} d e + 2 a c d e + a b e^{2}+ 2 a b d f \\
&-  c e^{2} -2 c d f   -2 b e f) + m_{2}^{2}( a^{3} b d^{2} -2 a b^{2} d^{2}-a^{2} c d^{2} -2 a^{2} b d e + a b d^{2}+ 2 b c d^{2}\\ 
&+ 2 a b d f+ 2 b^{2} d e + 2 a c d e +a b e^{2}-  c d^{2} -2 b e f -2 b d e -  c e^{2}  + 2 e f -2 c d f)\\
&+ m_{1} m_{2} (4a b^{2} d^{2}-2 a^{3} b d^{2}+2 a^{2} c d^{2}+ 4 a^{2} b d e -4 b c d^{2} -4 b^{2} d e -4a c d e -2 a b e^{2} \\
& -4a b d f+ 2 c e^{2}+ 2 b d + 4 c d f -2 f + 4 b e f) + n_{1}^{2}( a b d^{2}- c d^{2} -2 b d e+2ef)\\
& +  n_{1} n_{2}( -2 f + 2 b d) +m_{1} n_{1}( -2 b^{2} d^{2}-4a b d e+ 2 a^{2} b d^{2}  -2 a c d^{2} + 4 c d e+ 2 b e^{2}   \\
& + 4 b d f -2 f^{2}-2)+ 2 m_{1} n_{2}( a b d -cd - b e+1 )   \\
& + n_{1} m_{2} (2b^{2} d^{2}+ 4 a b d e + 2a c d^{2} -4 c d e -2 a^{2} b d^{2} -2 b e^{2} -4b d f + 2 f^{2} +2e)   \\
& +  m_{2} n_{2}(-2 a b d+2cd+2be -2e)-2 e f + 2 f - a b d^{2} + c d^{2}+ 2 b d e -2 b d\\
g_2(m_1,m_2,n_1,n_2)&=m_{1}^{2}( a^{4} d^{2} -3 a^{2} b d^{2} -2 a^{3} d e + 2 a c d^{2} + b^{2} d^{2} + 4 a b d e +  a^{2} e^{2} + 2a^{2} d f -2 c d e\\
&- b e^{2} -2  b d f -2 a e f   +f^{2}+1)+m_{2}^{2}( a^{4} d^{2} -3 a^{2} b d^{2} -2 a^{3} d e +b^{2} d^{2}+ a^{2} d^{2}\\
&+ 2 a c d^{2} + a^{2} e^{2}+ 4  a b d e - b d^{2} + 2 a^{2} d f -2 a d e+ -2 c d e-  b e^{2} -2 b d f -2 a e f \\
&+ e^{2} + 2 d f + f^{2} )+ m_{1} m_{2} (6a^{2} b d^{2} -2 a^{4} d^{2}  + 4 a^{3} d e - 2 b^{2} d^{2} -4 a c d^{2} -8 a b d e\\
&-2 a^{2} e^{2}-4 a^{2} d f+ 4 c d e + 2 b e^{2} + 4  b d f + 4 a e f + 2 a d+ -2 f^{2} -2 e)\\
&+ n_{1}^{2}( a^{2} d^{2} -  b d^{2} -2 a d e + e^{2} + 2 d f )+ n_{1} n_{2}(2 a d-2 e)+ n_{2}^{2}\\
&+m_{1} n_{1}( -4a b d^{2} - 4 a^{2} d e + 2 a^{3} d^{2} + 2  c d^{2}  + 4  b d e + 2 a e^{2}+ 4 a d f - 4 e f )    \\
&+  n_{1} m_{2}( 4a^{2} d e+ 4 a b d^{2}-2 a^{3} d^{2} -2 c d^{2} -4 b d e -4 a d f -2 a e^{2} + 4 e f+ 2 d )\\
&+  2 m_{1} n_{2}( a^{2} d -b d - a e + f )+ m_{2} n_{2} (-2 a^{2} d + 2 b d+ 2a e -2 d -2 f) \\
& - e^{2}  -2 d f   + 2 e - 1 - a^{2} d^{2} + 2 a d e + b d^{2} -2 a d. 
\end{align*}
We consider the matrix $M\in \Z[a,b,c,d,e,f]^{3\times11}$ whose three rows are the vectors of coefficients of $g_0,g_1,g_2$ with respect to the monomials $m_1^2,m_2^2,m_1m_2,n_1^2,n_2^2,n_1n_2, m_1n_1, m_2n_1,m_1n_2,m_2n_2$. For fixed $a,c,b,d,e,f\in\Q$, the three quadrics $Q_0,Q_1,Q_2$ defined respectively by
\begin{align*}
&g_0(a,b,c,d,e,f,m_1,m_2,n_1,n_2)=0\\
&g_1(a,b,c,d,e,f,m_1,m_2,n_1,n_2)=0\\
&g_2(a,b,c,d,e,f,m_1,m_2,n_1,n_2)=0
\end{align*}
are linearly dependent if and only if all the $3\times3$ submatrices of $M$ have determinant equal to zero. The minimal Gr\"obner basis for the ideal generated by these determinants, computed using SageMath~\cite{sagecode}, is as follows:
\begin{align*}
     \{c^4f^2 - 2c^3f + c^2f^2 + 2cf, a - c, b - 1, d - f, e - 1\}.
\end{align*}
Hence, if the quadrics are linearly dependent, $\theta_1$ satisfies the equation
$$0=\theta_1^3+a\theta_1^2+\theta_1+a=(\theta_1+a)(\theta_1^2+1),$$
that leads to a contradiction since $[\Q(\theta_1):\Q]=3$.

Now we consider the pencil described by the equation $g_0+\lambda g_1=0$. By \cite[Proposition 2.1]{Reid1972}, the variety $Q_0\cap Q_1$ is non-singular and of codimension 2 if and only if the polynomial $p(\lambda)=\det(A_0+\lambda A_1)$ is non-constant and has 5 distinct roots. Since $p$ is nonzero as a polynomial in $a,b,c,d,e,f,\lambda$, the values of $a,b,c,d,e,f$ for which $Q_0\cap Q_1$ is non-singular and of codimension 2 are Zariski-dense in $\Q^6$. For such values, the intersection of $Q_0, Q_1$ and $Q_2$ is a curve, and we can apply Proposition~\ref{ifcurve}.   

For $[Q(\theta_1,\theta_2):\Q]>3$, by dimensional consideration one should expect that the variety $V$ has dimension 1. Then, the conjecture should follow from Proposition~\ref{ifcurve}.

\section{The general case}\label{Secgen}
In this section, we move to general curves. The idea for studying Question~\ref{QuestionTorus} in the case of curves of high degree is to generalize the proof of Theorem~\ref{Twolinesandhyperbolic}. In that case, the triple intersections correspond to integral points, and we get the finiteness result by applying Siegel's theorem. If the curves have higher degrees, then the triple intersections turn out to be integral points of degree $d>1$ over the field of definition of the curves. The curves that potentially contain infinitely many integral points of degree $d$ are classified by the following theorem, which generalizes Siegel's theorem.

\begin{thm}[Levin, {\cite[Theorem 1.3]{Levin_2016}}]\label{Levin}
    Let $\cC\subset \A^n$ be a nonsingular affine curve defined over a number field $K$. Let $\widetilde{\cC}$ be a nonsingular projective completion of $\cC$ and let 
    $$(\widetilde{\cC}\setminus \cC)(\overline{K})=\{P_1,\dots,P_q\}.$$
    Let $d$ be a positive integer. Then there exists a finite extension $L$ of $K$ and a finite set of places $S$ of $L$ such that the set 
    $$\left\{P\in \cC(\overline{R_{L,S}})\mid [L(P):L]\leq d\right\}$$
    is infinite if and only if there exists a morphism $\phi\colon \widetilde{\cC}\rightarrow \pr^1$ over $\overline{K}$ with $\deg\phi\leq d$ and $\phi(\{P_1,\dots P_q\})\subset\{0,\infty\}$. 
\end{thm}

The following proposition follows directly from the previous theorem.

\begin{pro}\label{conLevin}
    Let $\cC_1, \cC_2$, and $\cC_3$ be three affine plane curves defined over a number field $K$, and let $d$ be an integer satisfying $\deg (\cC_1)\deg (\cC_2)\leq d$. Suppose that the three curves lie in different orbits under the action of $\Z^2$ over $\R^2$. Let $\widetilde{\cC_3}$ be a nonsingular completion of $\cC_3$ and and let 
    $$(\widetilde{\cC_3}\setminus \cC_3)(\overline{K})=\{P_1,\dots,P_q\}.$$ 
    Suppose that does not exist a morphism $\phi\colon \widetilde{\cC_3}\rightarrow \pr^1$ over $\overline{K}$ satisfying $\phi(\{P_1,\dots P_q\})\subset\{0,\infty\}$ and $\deg\phi\leq d$. Then $\cC_1, \cC_2$ and $\cC_3$ have only finitely many triple intersections on the real torus.
\end{pro}
\begin{proof}
     Let $\cC_1', \cC_2'$ be obtained by translating $\cC_1$ and $\cC_2$ respectively by an integer vector such that there exists a point of triple intersection $P\in \cC_1'\cap \cC_2'\cap \cC_3$. Then, since $P\in \cC_1'\cap \cC_2'$, $P$ has degree at most $d$ over $K$. Indeed, the action of $\mathrm{Gal(\overline{K}/K)}$ preserves the set $I:=\cC_1'(\overline{K})\cap \cC_2'(\overline{K})$, then all the conjugates of $P$ belong to $I$. Since, from Bezout's theorem, $I$ has cardinality at most $\deg (\cC_1)\deg (\cC_2)$, the point $P$ has degree at most $d$ over $K$. Moreover, we can choose a finite set $S$ of places of $K$ such that all the triple intersections of the three curves are $S$-integral. By Theorem~\ref{Levin}, $\cC_3$ has only finitely many $S$-integral points of degree less than or equal to $d$; therefore, there are only finitely many triple intersections on the real torus.
\end{proof}

The proposition shows that, if we fix $\cC_1$ and $\cC_2$, then for \virg{most of} the choices of $\cC_3$ the three curves have only finitely many triple intersections on the real torus. However, determining effectively whether a given curve has only finitely many integral points of degree $d$ is in general not easy. If the curve has $q>2d$ points at infinity, we can directly apply Levin's theorem, since there are no morphisms of degree less than or equal to $d$ that map a set of $q$ points to two points. In the other cases, we have to compute the gonality of the curve, but there are no efficient algorithms that solve this task.

\begin{oss}\label{remconLevin}
    Many examples that show that the hypothesis on the gonality of $\cC_3$ cannot be removed will be given in the following paragraph, where we specialize to the case of $\cC_1$ being a line. However, it should be possible to extend Theorem~\ref{conLevin} to a more general class of curves. For instance, the theorem cannot be applied to the case of two conics $\cC_1,\cC_2$ and a curve $\cC_3$ of degree 4 that has infinitely many integral points of degree 4. Heuristically, one should expect that only finitely many of them can be obtained as the intersection of a translate of $\cC_1$ with a translate of $\cC_2$.
\end{oss} 

Before moving to the case of one line and two curves of higher degree, we exhibit the following examples of three irreducible conics with infinitely many triple intersections on the real torus, which are of great interest because the intersections are explained with the action of an endomorphism of the torus that preserves the three curves. 

\subsection{Three conics with infinitely many triple intersections on the real torus.}

Let $\mathcal{P}$ be a pencil of conics such that there exists an automorphism of $\A^2$ represented by a matrix $T\in\mathrm{GL}(3,\Z)$ of infinite order satisfying $T(\cC)=\cC$ for each $C\in \mathcal{P}$. We remark that if the conics in $\mathcal{P}$ do not have a common component, then $\mathcal{P}$ has at most two base points. Let $\cC_1\in \mathcal{P}$ be any element of $\mathcal{P}$ and $P_1\in \cC_1$ be an irrational point. Given two integer vectors $v_2,v_3\in\Z^2$, we define the points $P_2:=P_1+v_2$ and $P_3:= P_1+v_3$, where a point of the plane is identified as a vector $(\alpha,\beta,1)^t\in\R^3$ and the translation vector as $v=(v^1,v^2,0)^t\in\R^3$. Let $\cC_2,\cC_3\in\mathcal{P}$ be the conics containing $P_2$ and $P_3$ respectively. Then the set 
$$\left\{T^n(P)\mid n\in\Z \right\}$$
is an infinite set of points of triple intersection on the torus. Indeed, 
$$T^n(P_1)-T^n(P_i)=T^n(v_i)\in\Z^2\quad \text{for }i=1,2\text{ and }n\in\Z.$$  

Now we construct some explicit examples.

\begin{es}[Hyperbolas]\label{threehyperbolas}
    We consider the pencil of hyperbolas
    $$\mathcal{P}:=\left\{ x^2-2y^2=c\mid c\in \R\right\}$$
    and the automorphism 
    $$T:=\begin{pmatrix}
        3 & 4 & 0\\
        2 & 3 & 0\\
        0 & 0 & 1
    \end{pmatrix}.$$
    We can easily check that $T(\cC)=\cC$  for each $\cC\in\mathcal{P}$. We take $\cC_1$ as the conic defined by $x^2-2y^2=1$, $P_1=:(\sqrt{3},1)\in \cC_1$, $P_2=(\sqrt{3},2)$ and $P_3=(\sqrt{3}+1, 1).$ As in the general construction, we consider the following conics of $\mathcal{P}$: 
    $$\cC_2: x^2-2y^2=-1\quad\text{and}\quad \cC_3: x^2-2y^2=2+2\sqrt{3}.$$
    Then $\cC_1, \cC_2, \cC_3$ have infinitely many points of triple intersection on the torus. 

    Now, let's take $\cC_1$ as the degenerate conic $x^2-2y^2=0$. The same argument provides us an example consisting of two hyperbolas and one line, since one between the two components of $\cC_1$, that is, $x=\sqrt{2}y$ and $x=-\sqrt{2}y$ must contain an infinite number of points of triple intersection. 
\end{es}

\begin{es}[Parabolas]\label{threeparabolas}
    We fix $a\in\R$, and consider the pencil of parabolas
    $$\mathcal{P}:=\left\{ ax^2+2y=c\mid c\in \R\right\}$$
    and the automorphism 
    $$T:=\begin{pmatrix}
        1 & 0 & -\frac{\lambda}{a}\\
        \lambda & 1 & -\frac{\lambda^2}{2a}\\
        0 & 0 & 1
    \end{pmatrix},$$
where $\lambda$ is a real parameter. We can easily check that $T(\cC)=\cC$  for each $C\in\mathcal{P}$. The matrix $T$ has integer coefficients if $a\in\Q^*$ and $\lambda\in 2a\Z\cap \Z$. For example, we can choose 
    $$a=\frac{1}{3},\quad \lambda=2,\quad \cC_1: 2y+\frac{1}{3}x^2=1,\quad P_1=(\sqrt{15},-2)$$ 
    and then construct two parabolas $\cC_2, \cC_3\in\mathcal{P}$ such that $\cC_1, \cC_2, \cC_3$ have an infinite number of points of triple intersection on the torus.
\end{es}

We remark that these points are obtained by the action of the automorphism $T$, similarly to Theorem~\ref{Homotheties}. Then, we can ask the following question.
\begin{que}
    Are there infinitely many orbits for the action of $\Gamma:=\langle \Z^2, T\rangle$ that intersect the three conics? 
\end{que}

\subsection{One line and two curves of higher degree}\label{seconeline}

Assuming that exactly one of the three curves is a line, we obtain the following stronger result.

\begin{pro}\label{lineandtwocurves}
    Let $\ell$ be a line and $\cC_1, \cC_2$ be two irreducible curves of degree $d_1< d_2$ respectively, all defined over a number field $K$. Suppose $d_2\geq 3$, let $\widetilde{\cC_1},\widetilde{\cC_2}$ be the projective completions of $\cC_1, \cC_2$ respectively and define $$(\widetilde{\cC_1}\setminus \cC_1)(\overline{K})=\{P_1,\dots,P_r\},\qquad (\widetilde{\cC_2}\setminus \cC_2)(\overline{K})=\{Q_1,\dots,Q_s\}.$$  If one of the following conditions hold, then $\ell, \cC_1, \cC_2$ have only finitely many triple intersections on the torus. 
    \begin{compactenum}
        \item $\widetilde{\cC_2}$ does not admit a morphism $\phi\colon \widetilde{\cC_2}\rightarrow \pr^1$ over $\overline{K}$ satisfying $\phi(\{Q_1,\dots Q_s\})\subset\{0,\infty\}$ and $\deg\phi\leq d_1$.
        \item The projective completion of $\ell$ does not meet neither $\widetilde{\cC_1}$ or $\widetilde{\cC_2}$, and there exists an integer 
        $$0\leq c \leq 2d_1-d_2-1$$ 
        such that $\widetilde{\cC_1}$ does not admit a morphism $\phi\colon \widetilde{\cC_1}\rightarrow \pr^1$ over $\overline{K}$ satisfying $\phi(\{P_1,\dots P_r\})\subset\{0,\infty\}$ and $\deg\phi\leq c$, and $\widetilde{\cC_2}$ does not admit a morphism $\phi\colon \widetilde{\cC_2}\rightarrow \pr^1$ over $\overline{K}$ satisfying $\deg\phi\leq d_1-1-c$ and $\phi(\{Q_1,\dots Q_s\})\subset\{0,\infty\}$.
    \end{compactenum}
\end{pro}
\begin{proof}
    The case (a) follows from Proposition~\ref{conLevin}. (b) Let $\ell', \cC_1'$ be obtained by translating $\ell$ and $\cC_1$ respectively by an integer vector such that there exists a point of triple intersection $P\in \ell'\cap \cC_1'\cap \cC_2$. As in the proof of Proposition~\ref{conLevin}, we can choose a finite set $S$ of places of $K$ such that $P$ is $S$-integral. Since $P\in \ell'\cap \cC_1'\cap \cC_2$, it has degree $d_P\leq d_1$. From Levin's Theorem~\ref{Levin}, there are only finitely many points on $\cC_2$ of degree at most $d_1-c-1$. Therefore, we can suppose $d_P\geq d_1-c$. If $c<d_P\leq d_1-1$, then there exists a point $Q\in \ell\cap \cC_1$ of degree $d_Q\leq c$, and again from Levin's Theorem~\ref{Levin} there are only finitely many of these points. If $P$ has degree $d_1$, then $\ell$ contains a point $R\in \ell\cap \cC_2$ of degree at most $d_2-d_1\leq d_1-c-1$. Hence, if a line $\ell'$ contains a triple intersection, then it contains an affine point belonging to a finite set. It follows that a triple intersection must lie on a finite number of lines; then there are only finitely many triple intersections.  
\end{proof}

Now we show that there exist configurations with infinitely many triple intersections consisting of a line, a curve of genus zero, and a curve of arbitrary genus. We use the following corollary of Hilbert's Irreducibility Theorem~\ref{Hilbert}.

\begin{pro}\label{corHilbert}
    Let $k$ be an integer, and $f(x)\in\Z[x]$ be a polynomial which is not a $k$-th power of a polynomial. Then, there are infinitely many integers $n$ such that $f(n)$ is not a $k$-th power of an integer. 
\end{pro}

\begin{es}[Curves with small gonality]\label{smallgon}
    Let $k$ be an integer, and $p(x)\in\Z[x]\setminus\{x\}$ be a polynomial which is not a $k$-th power of a polynomial. We consider the curves
    $$\ell: x=0,\quad C_1: y^k=x,\quad\text{and}\quad C_2: y^k=p(x).$$
    Then, for each $n\in\Z$, we set $Q_n:=(0,\sqrt[k]{p(n)})$ and get 
    $$Q_n\in \ell,\quad (p(n),\sqrt[k]{p(n)})=Q_n+(p(n),0)\in C_1,\quad (n,\sqrt[k]{p(n)})=Q_n+(n,0)\in C_2.$$
    Since for infinitely many $n$, $\sqrt[k]{p(n)}$ is not an integer, then the three curves have infinitely many intersections on the real torus.
\end{es}

Example \ref{smallgon} shows that we cannot weaken the condition on $\cC_2$ in Proposition~\ref{lineandtwocurves}.(a). Indeed, for each $h>k\geq 2$ we can consider the curves
$$\ell: x=0,\quad \cC_1: y^k=x,\quad\text{and}\quad \cC_2: y^k=x^h-x.$$
The curve $\widetilde{\cC_2}$ admits a map of degree $k$ to $\pr_1$ that maps its points at infinity to $\infty$, and the three curves have infinitely many triple intersections on the real torus.

\begin{oss}\label{remsymprod}
    In the proof of Proposition~\ref{lineandtwocurves}, we do not explicitly use the fact that every tuple of conjugate points must lie on a line of fixed slope. Indeed, if we have a line $\ell$, two curves $\cC_1$ and $\cC_2$, and a triple intersection $P$ of degree $n\leq \deg \cC_1$, then we have $n$ conjugate points on $\cC_2$ which lie on lines of a fixed slope. Therefore, we are looking for integral points on the symmetric product $\cC_2^{(n)}$ that satisfy an algebraic condition, meaning they lie on a subvariety $W$. Theorem~5.5 in \cite{Levin_2016} states that all but finitely many integral points of $\cC_2^{(n)}$ lie in positive-dimensional linear systems. Consequently, one expects $W$ to contain at least one subvariety isomorphic to the complement of a finite union of hyperplanes in a projective space.
    
    Example~\ref{smallgon} shows that, when the projective completion of $\ell$ and $\cC_2$ meet at a point at infinity $Q$ which has multiplicity $\deg \cC_2-n$, there are infinitely many $n$-tuples of conjugate points lying on a line with the same slope as $\ell$. The corresponding integral points on $\cC^{(n)}$ lie on a rational curve, parametrized by the pencil of lines containing $Q$. 
\end{oss}

\begin{oss}
    We can give another geometric characterization of the case of one line $\ell$ and two curves $\cC_1,\cC_2$ of higher degree having infinitely many triple intersections on the real torus. Let $\cC_1,\cC_2$ be defined by the polynomials $f_1(x,y),f_2(x,y)\in K[x,y]$ for some number field $K$. Suppose for simplicity that the line $\ell$ is the vertical axis given by the equation $x=0$, so that we have to consider only horizontal translations. Suppose that there are infinitely many $P=(0,y_P)\in\ell$ for which there exist integers $m_P,n_P\in\Z$ such that $(n_P,y_P)\in \cC_1$ and $(m_P,y_P)\in\cC_2$. Then, if we consider, in the space $\A^3$ with coordinates $u,v,w$, the curve cut out by the equations $f_1(u,w)=f_2(v,w)=0$; its projection $\widetilde\cC$ onto the first two coordinates has infinitely many integral points. Therefore, by Siegel's theorem, there is a component of $\widetilde\cC$ which is rational and admits a normalization with at most two points at infinity. We observe that this phenomenon happens in Example~\ref{smallgon}, where the curve $\widetilde\cC$ has equation $u=p(v)$. Classifying all pairs $\cC_1,\cC_2$ such that $\widetilde\cC$ has a rational component is a difficult problem, but a similar idea is used in the next part of this section to deal with the case of two curves with the same degree.
\end{oss}

Proposition~\ref{lineandtwocurves} does not cover the case of one line and two curves of the same degree. We address this situation in the following theorem.

\begin{thm}\label{lineandsamedegree}
    Let $\ell\subset \A^2(\R)$ be a line defined over $\Q$ and $\cC_1,\cC_2\subset \A^2(\R)$ be two curves of degree $d>2$, defined over a number field $K$, that do not intersect $\ell$ at infinity. Let $\widetilde{\cC_2}$ be a nonsingular projective completion of $\cC_2$ and let 
    $$(\widetilde{\cC_2}\setminus \cC_2)(\overline{K})=\{P_1,\dots,P_q\}.$$ Suppose that does not exist a morphism $\phi\colon \widetilde{\cC_2}\rightarrow \pr^1$ over $\overline{K}$ satisfying $\phi(\{P_1,\dots P_q\})\subset\{0,\infty\}$ and $\deg\phi\leq\lfloor\frac{d}{2}\rfloor$. If $\ell, \cC_1, \cC_2$ have infinitely many triple intersections on the real torus, then there exist two endomorphisms $\varphi_1,\varphi_2$ of $\A^2$ that map respectively $\cC_1$ and $\cC_2$ to a curve $\cC\subset \A^2$ defined over $\Q$. The same result holds when switching the roles of $\cC_1$ and $\cC_2$.
\end{thm}

Let $\ell, \cC_1,\cC_2\subset \A^2(\R)$ be as in the Theorem~\ref{lineandsamedegree}. Up to a change of variables over $\Q$, we can reduce to the following case:  
$$\ell:x=0\quad \cC_1:y^d+\sum_{i=0}^{d-1}a_i(x)y^i=0\quad \cC_2:y^d+\sum_{i=0}^{d-2}b_i(x)y^i=0,$$
satisfying $a_i(x),b_i(x)\in\ K[x]\setminus\{0\}$, where $\deg a_i,\deg b_i\leq d-i$ for each $i$.

If there are infinitely many triple intersections, then there is a solution of the following system for infinitely many $n,m,k\in \Z$
$$\begin{cases}
    x=m\\
    (y+k)^d+\sum_{i=0}^{d-1}a_i(x+n)(y+k)^i=0\\
    y^d+\sum_{i=0}^{d-2}b_i(x)y^i=0,
\end{cases}$$
which is equivalent to
$$\begin{cases}
    p_1(l,k,y)&:=(y+k)^d+\sum_{i=0}^{d-1}a_i(l)(y+k)^i\\
    &\;=y^d+y^{d-1}(dk+a_{d-1}(l))+\sum_{i=0}^{d-2}q_i(k,l)y^i=0\\
    p_2(m,y)&:= y^d+\sum_{i=0}^{d-2}b_i(m)y^i=0,
\end{cases}$$
where we have defined $l:=m+n$ and $q_i(k,l):=\binom{d}{j}k^{d-j}+\sum_{j=i}^{d-2}\binom{j}{i}a_j(l)k^{j-i}$.

From the hypothesis and Theorem~\ref{Levin}, the curve $\cC_2$ has only finitely many integral points of degree at most $\lfloor\frac{d}{2}\rfloor$. Hence, if there are infinitely many triple intersections, they correspond to points of degree $d$. Indeed, if $\ell\cap \cC_2$ is reducible over $K$, at least one component has degree at most $\lfloor\frac{d}{2}\rfloor$. The equations $p_1(l,k,y)=0,p_2(m,y)=0$ in $y$ have the same solution of degree $d$ over $\Q$ if and only if the polynomials $p_1(l,k,y),p_2(m,y)\in\Q[y]$ have the same roots, that is, since they are monic, if and only if they are equal. This implies that for infinitely many integers $l,m,k$ we have

$$\begin{cases}
    dk+a_{d-1}(l)=0\\
    q_{d-2}(k,l)=b_{d-2}(m)\\
    \dots\\
    q_{0}(k,l)=b_0(m).
\end{cases}$$

Eliminating $k$, we get the following system.

\begin{equation}\label{dcurve}
    \begin{cases}
        q_{d-2}\left(-\frac{a_{d-1}(l)}{d},l\right)=b_{d-2}(m)\\
        \dots\\
        q_{0}\left(-\frac{a_{d-1}(l)}{d},l\right)=b_0(m).
\end{cases}
\end{equation}
which have infinitely many integral solutions if and only if the curves defined by the equations have a common component that has infinitely many integral points. By Siegel's Theorem, that curve must have genus zero and at most two points at infinity.

\begin{defi}
    Let $K$ be a field. An absolutely irreducible polynomial $E(x, y) \in K[x, y]$ is called \emph{exceptional} if the plane curve $E(x, y) = 0$ is of genus 0 and has at most two points at infinity.
\end{defi}

In \cite{BT2000}, there is a classification of polynomials of the form $f(x)-g(y)\in K[x,y]$, where $K$ is a number field, that have exceptional factors. Before stating the theorem, we introduce some terminology.

\begin{defi}[Dickinson's polynomials]
    Let $K$ be a field. For $n\in\Z_{\geq 1}$ and $a\in K$, the $n$-th Dickinson's polynomial $D_n(x,a)$ is defined by
    $$D_n(x,a)=\left(\frac{x+\sqrt{x^2-4a}}{2}\right)^n+\left(\frac{x-\sqrt{x^2-4a}}{2}\right)^n.$$
\end{defi}

\begin{defi}[Standard pairs]
    Let $K$ be a field. Let $a,b\in K^*$, $m,n$ be positive integers, and $p(x)\in K[x]\setminus\{0\}$. We define:
    \begin{compactenum}[(i)]
        \item a \emph{standard pair of the first kind} as 
        $$(x^m,ax^rp(x)^m)$$
        where $0\leq r<m,$ $(r,m)=1$ and $r+\deg p(x)>0$;
        \item a \emph{standard pair of the second kind} as 
        $$(x^2,(x^2+b)p(x)^2);$$
        \item a \emph{standard pair of the third kind} as 
        $$\left(D_m(x,a^n),D_n(x,a^m)\right)$$
        where $(m,n)=1$ and $D_n$ denotes the $m$-th Dickinson polynomial, defined by
        $$D_m\left(z + \frac{a}{z}, a\right) = z^m + \left(\frac{a}{z}\right)^m;$$
        \item a \emph{standard pair of the fourth kind} as 
        $$\left(a^{-\frac{m}{2}}D_m(x,a),b^{-\frac{n}{2}}D_n(x,b)\right)$$
        where $(m,n)=2$;
        \item a \emph{standard pair of the fifth kind} as 
        $$((ax^2-1)^3, 3x^4-4x^3).$$
        In all five cases, if we switch the components of the pairs, we obtain a standard pair of the same kind. A generic pair of the previous list will be called \emph{standard pair over $K$}. The \emph{degree} of a standard pair $(f,g)$ is the $d:=\max\{\deg f,\deg g\}.$  
    \end{compactenum}
\end{defi}

We observe that we can reduce to the case $K=\Q$, since the curves defined by the Equations (\ref{dcurve}) have infinitely many $\Z$-integral points, and so are defined over $\Q$.
\begin{thm}[{\cite[Theorems 1.1, 9.3]{BT2000}}]\label{BiluTichy}
Let $f(x), g(x) \in \Q[x]$ be non-constant polynomials. Then the following two assertions are equivalent.
    \begin{compactenum}[(a)]
        \item The equation $f(x)-g(y)=0$ has infinitely many $\Q$-rational solutions with a bounded denominator.
        \item We have $f = \phi \circ f_1 \circ \lambda$ and $g = \phi \circ g_1 \circ \mu,$ where $\lambda(x), \mu(x) \in \Q[x]$ are linear polynomials, $\phi(x) \in \Q[x]$, and $(f_1(x), g_1(x))$ is a standard pair over $\Q$ such that the equation $f_1(x) = g_1(y)$ has infinitely many rational solutions with a bounded denominator. 
    \end{compactenum}
    Moreover, if the assertions hold, then the polynomial $f(x)-g(y)$ has the exceptional factor $E(x,y)=c(f_1(\lambda(x))-g_1(\mu(y)))$ where $c\in\Q$ is a constant.
\end{thm}

Applying Theorem~\ref{BiluTichy} to Equations~(\ref{dcurve}), we obtain that, if there exist infinitely many triple intersections between the three curves $\ell, \cC_1, \cC_2$, then there exist polynomials  $$\lambda(x),\mu(x),\varphi_0(x),\dots,\varphi_{d-2}(x)\in\Q[x]\quad\text{with }\deg\lambda=\deg\mu=1$$
and a standard pair $(f,g)$ with $f(x),g(x)\in\Q[x]$ satisfying, for each $1\leq i\leq d-2$ 
$$q_{i}\left(-\frac{a_{d-1}(x)}{d},x\right)=\varphi_i(f(\lambda(x)))\quad\text{and}\quad b_{i}(x)=\varphi_i(g(\mu(x))).$$

Let $\cC$ be the curve defined by the equation $y^d+\varphi_{d-2}(x)y^{d-2}+\dots+\varphi_0(x)=0$. The curves $\cC_1$ and $\cC_2$ can be mapped to the curve $\cC$ applying the endomorphisms of the plane $\psi_1, \psi_2$ respectively, defined as follows:
$$\psi_1\colon (x,y)\mapsto \big(f(x),y+\frac{1}{d}a_{d-1}(x)\big)\quad \psi_2\colon (x,y)\mapsto \big(g(x),y\big).$$

This concludes the proof of Theorem~\ref{lineandsamedegree}. An analogous result should be expected even without assuming that the common solutions of Equations~\eqref{dcurve} are integral. Indeed, if there are infinitely many solutions, the plane curves defined by Equations~\eqref{dcurve} have a common component, and this fact should give strong restrictions on the curves $\cC_1,\cC_2$. A more general version of the problem is the following:

 \begin{que}\label{quecomplexpoints}
     Let $\cC_1, \cC_2\subset \A^2(\C)$ be two affine plane curves, and $\ell\subset \A^2(\C)$ be a line. Suppose that there exist infinitely many pairs of vectors $(v_1,v_2)\in \C^2$ with the following property: the line $\ell_{v_1}$ obtained translating $\ell$ by $v_1$ intersect $\cC_1$ in $d$ complex points, and the translations of this $d$ points by $v_2$ all lie on $\cC_2$. What can we say about $\cC_1$ and $\cC_2$?    
 \end{que}

 We expect that two curves with that property should satisfy a geometric condition like the one arising in Theorem~\ref{lineandsamedegree}. However, the proof of that result should rely only on geometric considerations and avoid arithmetic ones.

 \subsubsection*{An example: $d=3$.}
 Due to Theorem~\ref{lineandsamedegree}, for classifying pairs of cubic curves $\cC_1, \cC_2$ such that $\ell, \cC_1, \cC_2$ have infinitely many triple intersections, we have to classify standard pairs of degree at most 3.

 \begin{compactenum}[(i)]
    \item Standard pairs of the first kind of degree 2:
    \begin{itemize}
        \item case 1: $m=1$, $r=0$, $\deg p\leq 2$, then $f(x)=x$, $g(x)=ax^2+bx+c$, $a,b,c\in\Q, a\neq 0$;
        \item case 2: $m=2$, $r=0$, $\deg p=1$, then $f(x)=x^2$, $g(x)=a(bx+c)^2$ with $a,b,c\in\Q, ab\neq 0$.
    \end{itemize}  
    \item Standard pairs of the first kind of degree 3:
    \begin{itemize}
        \item case 1: $m=1$, $r=0$, $\deg p =3$, then $f(x)=x$, $g(x)=ax^3+bx^2+cx+d$, with $a,b,c,d\in\Q,\; a\neq 0$;
        \item case 2: $m=2$, $r=1$, $\deg p =1$, then $f(x)=x^2$, $g(x)=ax(bx+c)^2$ with $a,b,c\in\Q,\; ab\neq 0$;
        \item case 3: $m=3$, $r=0$, $\deg p =1$, then $f(x)=x^3$, $g(x)=a(bx+c)^3$ with $a,b,c\in\Q,\; ab\neq 0$;
    \end{itemize}
    \item Standard pairs of the second kind of degree 2: $f(x)=x^2$, $g(x)=ax^2+b$ with $a,b\in\Q$, $a\neq 0$.
    \item Standard pairs of the third kind of degree 2: the only case is $m=2, n=1$, then $f(x)=x^2-2a$ with $a\in\Q$, $g(x)=x$.  
    \item Standard pairs of the third kind of degree 3: 
    \begin{itemize}
        \item case 1: $m=3,\; n=2$, then $f(x)=x^3-3ax$ with $a\in\Q$, $g(x)=x$. 
        \item case 2: $m=3,\; n=1$, then $f(x)=x^3-3a^2x$ with $a\in\Q$, $g(x)=x^2-2a^3$. 
    \end{itemize}
    \item Standard pairs of the fourth kind of degree 2: the only case is $m=n=2$, then $f(x)=a^{-1}(x^2-a),$ and $g(x)=-b^{-1}(x^2-b)$.
    \item There are no standard pairs of degree 2 of the fifth kind, and there are no standard pairs of degree 3 of the second, fourth, and fifth kind.
\end{compactenum}
To sum up, exceptional factors of degree 2, up to a multiplicative constant, are all of the form
    $$ax+b-(cy^2+dy+e) \quad \text{or}\quad (ax+b)^2-(cy+d)^2-e$$
with $a,b,c,d,e\in\Q, a,c\neq 0$. Meanwhile, exceptional factors of degree 3 are all of the form
\begin{align*}
    &ax+b- (cy^3+dy^2+ey+f),\quad (ax+b)^2-cy(dy+e)^2\\
    &(ax+b)^3-c(dy+e)^3,\quad\text{or}\quad (bx+c)^3-3a^2(bx+c)-(dy+e)^2+2a^3.
\end{align*}
We summarize the discussion in the following example.
\begin{es}\label{linetwocub}
    The pairs of cubic curves $(\cC_1, \cC_2)$ defined over $\Q$ and having infinitely many triple intersections on the real torus with the line $\ell:= x=0$ are, up to linear change of coordinates defined over $\Q$, the following:
    \begin{compactenum}[(a)]
        \item $\cC_1: y^3+p(x)y+p(x)=0$ and $\cC_2: y^3+q(x)y+q(x)$, where $p(x)\in\Q[x]$ has degree 1 and $q(x)\in\Q[x]$ has degree 1 or 2.
        \item $\cC_1: y^3+p(x)^2y+p(x)^2=0$ and $\cC_2: y^3+q(x)^2y+q(x)^2=0$, where $p(x),q(x)\in\Q[x]$ have degree 1.
        \item $\cC_1: y^3+p(x)=0$ and $\cC_2: y^3+q(x)$, where $p(x)\in\Q[x]$ has degree 1 and $q(x)\in\Q[x]$ has degree 1,2 or 3.
        \item $\cC_1: y^3+p(x)^2=0$ and $\cC_2: y^3+q(x)^2=0$, where $p(x),q(x)\in\Q[x]$ have degree 1.
        \item $\cC_1: y^3+p(x)^2=0$ and $\cC_2: y^3+axq(x)^2=0$, where $p(x),q(x)\in\Q[x]$ have degree 1 and $a\in\Q$.
        \item $\cC_1: y^3+p(x)^3=0$ and $\cC_2: y^3+q(x)^3$, where $p(x),q(x)\in\Q[x]$ have degree 1.
        \item $\cC_1: y^3+p(x)^3-3a^2p(x)=0$ and $\cC_2: y^3+q(x)^2-2a^3$, where $p(x),q(x)\in\Q[x]$ have degree 1 and $a\in\Q$.
    \end{compactenum}
\end{es}

\section{Reduction to number fields}\label{Secspecialization}
In this section, we present a specialization argument that extends the results established in the previous sections for curves defined over a number field to generic real curves.

Let $\cC_1, \cC_2, \cC_3$ be three algebraic plane curves, and $K$ be a finitely generated field over $\Q$ where the three curves are defined. Let $w_1,\dots,w_m$ be a transcendence basis for the extension $K/\Q$. By the primitive element theorem, there exists $\gamma\in K$ such that 
$K=\Q(w_1,\dots,w_m)[\gamma]$ where $\gamma$ is algebraic over $\Q(w_1,\dots,w_m)$. 

The curves $\cC_1,\cC_2$, and $\cC_3$ are described by equations of the following form:
$$\cC_i: \sum_{j,k} F_{j,k}^{(i)}(w_1,\dots,w_m,\gamma)x^jy^k=0,$$
where $F_{j,k}^{(i)}(w_1,\dots,w_m,\gamma)\in\Q[w_1,\dots,w_m](\gamma)$. We need a specialization 
$$\phi: (w_1,\dots,w_m)\mapsto (\alpha_1,\dots,\alpha_m)\in (\overline{\Q})^m$$
with the following properties:
\begin{compactenum}[(i)]
    \item $\phi$ maps distinct curves to distinct curves;
    \item $\phi$ preserves the degree, the genus and the gonality of each curve;
    \item $\phi$ maps curves not in the same orbit under the action of $\Z^2$ to curves with the same property.
\end{compactenum}
Fixed $\alpha_1,\dots,\alpha_m,\gamma$, the property (i) does not hold when there are $i,i'$ such that 
$$F_{j,k}^{(i)}(\alpha_1,\dots,\alpha_m,\gamma)=F_{j,k}^{(i')}(\alpha_1,\dots,\alpha_m,\gamma)\quad\text{for each } j,k.$$
Moreover, the property (ii) fails when some algebraic conditions on the coefficients are satisfied, that is, some polynomials vanish in $\alpha_1,\dots,\alpha_m,\gamma$. Hence, the $m$-uples $(\alpha_1,\dots,\alpha_m)$ for which properties (i) and (ii) do not hold are contained in a Zariski-closed proper subset of $\overline{\Q}^m$. Therefore, there are infinitely many specializations which satisfy properties (i) and (ii).

Now we have to prove that there are specializations that map two curves that do not belong to the same orbit with respect to the action of $\Z^2$ in curves with the same property. We consider two curves $\cC_1(w_1,\dots,w_m), \cC_2(w_1,\dots,w_m)$ as before, where we have made explicit the dependence on $w_1,\dots,w_m$. Let us consider the set
$$X=\left\{(w_1,\dots,w_m,v)\in\A^m\times \A^2\bigm| \cC_1(w_1,\dots,w_m)+v=\cC_2(w_1,\dots,w_m)\right\},$$
where the notation $\cC_1(w_1,\dots,w_m)+v$ represent the set obtained translating the points of $\cC_1$ by the vector $v$. We say that the $m$-tuple $\alpha_1,\dots,\alpha_m$ is a \emph{bad choice} if there exists $v\in\Z^2$ satisfying 
\begin{equation}\label{tras}
\cC_1(\alpha_1,\dots,\alpha_m)+v=\cC_2(\alpha_1,\dots,\alpha_m).
\end{equation}
If there are only finitely many bad choices, we can choose a specialization that satisfies property (iii). Now we suppose that there are infinitely many bad choices. Then $\dim X>0$, that is, $X$ is an algebraic variety and we have a projection $X\rightarrow \A^m$. We observe that for each $m$-tuple $\alpha_1,\dots,\alpha_m$ there is at most one vector $v\in \A^2$ satisfying Condition~\eqref{tras}, hence we have a section 
$$f\colon \A^m\rightarrow X \quad (w_1,\dots,w_m)\mapsto ((w_1,\dots,w_m),v(w_1,\dots,w_m)),$$
and then $v(w_1,\dots,w_m)$ is a rational function of $w_1,\dots,w_m$. Therefore, it is enough to choose a specialization such that $v(\alpha_1,\dots,\alpha_m)$ is not an integer vector. 

\bibliographystyle{plain}
\bibliography{bibliografia}

\end{document}